\documentclass{article}
\usepackage{amsmath,amssymb,amsthm,color,graphicx,diagbox,tikz,colortbl,array}
\usepackage{setspace}
\usepackage{float}

\title{
Modal logical aspects of provability predicates and consistency statements
}
\author{Haruka Kogure\footnote{Email: kogure@stu.kobe-u.ac.jp}
\footnote{Graduate School of System Informatics, Kobe University, 1-1 Rokkodai, Nada, Kobe 657-8501, Japan.}
and Taishi Kurahashi\footnote{Email: kurahashi@people.kobe-u.ac.jp}
\footnote{Graduate School of System Informatics, Kobe University, 1-1 Rokkodai, Nada, Kobe 657-8501, Japan.}}
\date{}

\theoremstyle{plain}
\newtheorem{thm}{Theorem}[section]
\newtheorem*{thm*}{Theorem}

\newtheorem{prop}[thm]{Proposition}
\newtheorem{cor}[thm]{Corollary}

\newtheorem*{fact*}{Fact}

\newtheorem*{prob*}{Problem}
\newtheorem*{cl*}{Claim}
\newtheorem{cl}{Claim}[section]
\newtheorem*{scl*}{Subclaim}

\theoremstyle{definition}
\newtheorem{defn}[thm]{Definition}

\newcommand{\PL}{\mathsf{PL}}
\newcommand{\PRR}{\mathrm{Pr}^{\mathrm{R}}}

\newcommand{\Sub}{\mathsf{Sub}}
\newcommand{\MF}{\mathsf{MF}}
\newcommand{\N}{\mathsf{N}}
\newcommand{\tc}{\vdash^{\mathrm{t}}}
\newcommand{\Img}{\mathrm{Im}}

\newcommand{\PA}{\mathsf{PA}}
\newcommand{\PR}{\mathrm{Pr}}
\newcommand{\Prf}{\mathrm{Prf}}
\newcommand{\Prov}{\mathrm{Prov}}
\newcommand{\Proof}{\mathrm{Proof}}
\newcommand{\Con}{\mathrm{Con}}

\newcommand{\gn}[1]{\ulcorner#1\urcorner}

\newcommand{\D}[1]{\mathbf{D#1}}

\newcommand{\num}{\overline}

\newcommand{\LA}{\mathcal{L}_A}

\newcommand{\ConH}{\mathrm{Con}^{\mathrm{H}}}
\newcommand{\ConS}{\mathrm{Con}^{\mathrm{S}}}

\newcommand{\ConL}{\mathrm{Con}^{\mathrm{L}}}

\newcommand{\NPF}{\mathsf{NP4}}
\newcommand{\NDF}{\mathsf{ND4}}
\newcommand{\NP}{\mathsf{NP}}
\newcommand{\ND}{\mathsf{ND}}
\newcommand{\GL}{\mathsf{GL}}

\newcommand{\PRD}{\mathrm{Pr}^{\mathrm{A}}}

\begin{document}

\maketitle

\begin{abstract}
This paper studies the modal logical aspects of provability predicates and consistency statements for theories of arithmetic.
First, we provide an overview of previous works on the correspondence between various derivability conditions for provability predicates and different modal logics.
The main technical contribution of the present paper is to establish the arithmetical completeness of the logics $\mathsf{NP}$, $\mathsf{ND}$, $\mathsf{NP4}$, and $\mathsf{ND4}$ by extending Solovay's method and refining Arai's construction of Rosser provability predicates.
\end{abstract}

\section{Introduction}

This paper is part of a research project on the modal logical studies of derivability conditions for provability predicates and various formulations of consistency statements.
Let $T$ be a primitive recursively axiomatized consistent extension of Peano Arithmetic $\PA$.
G\"odel's second incompleteness theorem (G2) states that $T$ cannot prove the sentence $\Con_T$ expressing its own consistency \cite{Goedel}. 
However, in order to state and prove G2 precisely, one must carefully specify how the consistency statement $\Con_T$ is formulated in terms of a provability predicate $\PR_T(x)$. 
Also, the exact form of G2 depends on how the provability predicate is defined and on which derivability conditions it satisfies.
Hilbert and Bernays \cite{HB} provided the first detailed proof of G2 by presenting a set of conditions for provability predicates that suffice for their proof. 
Later, L\"ob \cite{Loeb} refined them to the following well-known conditions $\D{1}$, $\D{2}$, and $\D{3}$, currently called the Hilbert–Bernays–L\"ob derivability conditions:
\begin{description}
    \item [$\D{1}$]: $T \vdash \varphi \Rightarrow T \vdash \PR_T(\gn{\varphi})$. 
    \item [$\D{2}$]: $T \vdash \PR_T(\gn{\varphi \to \psi}) \to (\PR_T(\gn{\varphi}) \to \PR_T(\gn{\psi}))$. 
    \item [$\D{3}$]: $T \vdash \PR_T(\gn{\varphi}) \to \PR_T(\gn{\PR_T(\gn{\varphi})})$. 
\end{description} 
If $\PR_T(x)$ satisfies these conditions, then the consistency statement $\ConL_T : \equiv \neg \PR_T(\gn{0=1})$ is unprovable in $T$. 
This is the familiar version of G2. 
Moreover, these conditions yield the formalized version of L\"ob's theorem: $T \vdash \PR_T(\gn{\PR_T(\gn{\varphi}) \to \varphi}) \to \PR_T(\gn{\varphi})$. 
In fact, the naturally constructed provability predicate $\Prov_T(x)$ of $T$ satisfies the conditions $\D{1}$–$\D{3}$. 
\emph{Provability logic} is the research area that studies provability predicates by means of modal logic. 
The modal logic $\GL$ is axiomatized by the axioms and rules corresponding to $\D{1}$-$\D{3}$ together with the formalized version of L\"ob's theorem. 
Solovay's arithmetical completeness theorem \cite{Solovay} states that $\GL$ precisely captures all $T$-verifiable principles concerning $\Prov_T(x)$, provided $T$ is $\Sigma_1$-sound.

The version of G2 based on $\D{1}$–$\D{3}$ and $\ConL_T$ is not the only possible one. 
Actually, Hilbert and Bernays dealt with an alternative formulation of the consistency statement $\ConH_T : \equiv \forall x \, \neg (\PR_T(x) \land \PR_T(\dot{\neg} x))$ and proved G2 under different sets of derivability conditions.
Here, $\dot{\neg}$ is a primitive recursive term corresponding to a primitive recursive function calculating the G\"odel number of $\neg \varphi$ from that of $\varphi$. 
Jeroslow \cite{Jeroslow} and Kreisel and Takeuti \cite{KT} later showed that formalized $\Sigma_1$-completeness suffices for the unprovability of $\ConH_T$, among other results. 
These developments indicate that G2 is not a single statement, but rather a family of related theorems stating the unprovability of various consistency statements under appropriate derivability conditions on $\PR_T(x)$. 

In recent years, several studies have investigated the modal logics corresponding to various sets of derivability conditions and formulations of consistency statements. 
In particular, the papers \cite{Kura18,Kura18b,Kura20b,KK,Kura24,Kogure1,Kogure2} have established the arithmetical completeness of several modal logics. 
The present paper is situated within this line of research.  
The aim of this study is twofold.  
The first aim is to give a unified account of previous results on the modal logical investigation of G2 based on various derivability conditions and consistency statements (discussed in Section~\ref{Summary}). 
Through this overview, we clarify how the principles $\mathbf{E}$, $\mathbf{M}$, $\mathbf{C}$, and $\D{3}$, together with the consistency statements $\ConL_T$ and $\ConS_T := \{\neg \,(\PR_T(\gn{\varphi}) \land \PR_T(\gn{\neg \varphi})) \mid \varphi$ is a sentence$\}$, correspond to a wide hierarchy of modal logics, each representing a distinct combination of derivability conditions. 
Table \ref{Table1} in the next section summarizes the overall situation of known and new results.

The starting point for this hierarchy is Fitting, Marek, and Truszczy\'nski's pure logic of necessitation $\N$ \cite{FMT}. 
The second author proved in \cite{Kura24} that $\N$ is the provability logic of all provability predicates by using the relational semantics developed by Fitting et al. 
Moreover, it is proved that the logics $\mathsf{N4}$, $\mathsf{NR}$, and $\mathsf{NR4}$ obtained from $\N$ by adding at least one of the axiom $\mathsf{4}$: $\Box A \to \Box \Box A$ and the Rosser rule \textsc{Ros}: $\dfrac{\neg A}{\neg \Box A}$ are also arithmetically complete.  
The second aim of this paper is to prove the arithmetical completeness for the logics $\NP$, $\ND$, $\NPF$, and $\NDF$. 
The proof proceeds by first establishing the modal completeness and the finite frame property of these logics with respect to the relational semantics due to Fitting et al. 
Here, a modal logic $L$ is said to have the finite frame property if for any formula $A$, $A$ is provable in $L$ if and only if $A$ is valid in all finite frames validating all theorems of $L$.
Then, we extend the Solovay-style method to obtain  arithmetical completeness for the logics $\NP$, $\ND$, $\NPF$, and $\NDF$.
The case of $\NDF$, in particular, requires a crucial refinement of Arai's construction of Rosser provability predicates \cite{Arai}.

\section{Summary of previous works}\label{Summary}

This section provides an overview of our previous studies on the modal logical analysis of provability predicates and consistency statements. 
The aim of this section is twofold: to clarify how different derivability conditions correspond to modal principles, and to identify the logics whose arithmetical completeness remains unsettled.

\subsection{Provability predicates, consistency statements, and derivability conditions}

Let $T$ be a primitive recursively axiomatized consistent extension of Peano Arithmetic $\PA$ in the language $\LA$ of first-order arithmetic 
(see H\'ajek and Pudl\'ak \cite{HP} for the background of first-order arithmetic).
A formula $\PR_T(x)$ is called a \emph{provability predicate} of $T$ if for any $\LA$-formula $\varphi$, $T \vdash \varphi$ if and only if $\PA \vdash \PR_T(\gn{\varphi})$. 
As stated in the introduction, the derivability conditions $\D{1}$–$\D{3}$ suffice to prove the $T$-unprovability of the consistency statement $\ConL_T :\equiv \neg \PR_T(\gn{0=1})$. 
However, $\ConL_T$ is not the only reasonable formulation of consistency.  
In \cite{Kura20}, a systematic classification of derivability conditions and the corresponding versions of G2 was presented.
Based on this work, \cite{Kura25} further refined the framework by introducing additional derivability conditions and formulations of consistency statements. 

First, in addition to $\ConL_T$, we introduce the following two consistency principles:
\begin{itemize}
\item $\mathbf{Ros}$: If $T \vdash \neg \varphi$, then $T \vdash \neg \,\PR_T(\gn{\varphi})$,
\item $\ConS_T := \{\neg \,(\PR_T(\gn{\varphi}) \land \PR_T(\gn{\neg \varphi})) \mid \varphi$ is a sentence$\}$. 
\end{itemize}
Here, $\mathbf{Ros}$ is the \emph{Rosser rule}, and $\ConS_T$ is the \emph{schematic consistency statement}. 
The principle $\ConS_T$ was introduced in \cite{Kura21}. 

For a provability predicate $\PR_T(x)$ and an $\LA$-formula $\varphi$, we recursively define $\PR_T^n(\gn{\varphi})$ as follows: $\PR_T^0(\gn{\varphi}) \equiv \varphi$ and $\PR_T^{n+1}(\gn{\varphi}) \equiv \PR_T(\gn{\PR_T^n(\gn{\varphi})})$. 
Next, we introduce the following derivability conditions: 
\begin{itemize}
\item $\D{3}^n_m$: $T \vdash \PR_T^n(\gn{\varphi}) \to \PR_T^m(\gn{\varphi})$,
\item $\mathbf{E}$: If $T \vdash \varphi \leftrightarrow \psi$, then $T \vdash \PR_T(\gn{\varphi}) \leftrightarrow \PR_T(\gn{\psi})$,
\item $\mathbf{M}$: If $T \vdash \varphi \to \psi$, then $T \vdash \PR_T(\gn{\varphi}) \to \PR_T(\gn{\psi})$,
\item $\mathbf{C}$: $T \vdash (\PR_T(\gn{\varphi}) \land \PR_T(\gn{\psi})) \to \PR_T(\gn{\varphi \land \psi})$.
\end{itemize}
In particular, the conditions $\mathbf{E}$, $\mathbf{M}$, and $\mathbf{C}$ correspond to well-known principles in the context of non-normal modal logics, and it was shown in \cite{Kura25} that these conditions play a crucial role in establishing G2.
Specifically, the following statements were proved:

\begin{thm}[Kurahashi \cite{Kura25}]\label{Kura}
Let $m > n \geq 1$. 
\begin{enumerate}
\item If $T \vdash \ConS_T$, then $\mathbf{Ros}$ holds. Conversely, if $\PR_T(x)$ satisfies $\mathbf{C}$, then $\mathbf{Ros}$ implies $T \vdash \ConS_T$.
\item If $\mathbf{Ros}$ holds, then $T \vdash \ConL_T$. Conversely, if $\PR_T(x)$ satisfies $\mathbf{E}$, then $T \vdash \ConL_T$ implies $\mathbf{Ros}$.
\item If $\PR_T(x)$ satisfies $\mathbf{E}$ and $\D{3}$, then $T \nvdash \ConS_T$.
\item If $\PR_T(x)$ satisfies $\mathbf{C}$ and $\D{3}^n_m$, then $\mathbf{Ros}$ fails to hold.
\item If $\PR_T(x)$ satisfies $\mathbf{E}$, $\mathbf{C}$, and $\D{3}^n_m$, then $T \nvdash \ConL_T$.
\end{enumerate}
\end{thm}

\subsection{Modal logics associated with derivability conditions}

These derivability conditions correspond to particular modal axioms and rules, and this correspondence forms a hierarchy of modal logics.  
We have been investigating these modal logics focusing on their arithmetical completeness. 
Let $\PR_T(x)$ be a provability predicate of $T$. 
A mapping $f$ from the set of all modal formulas to a set of $\LA$-sentences is called an \emph{arithmetical interpretation} based on $\PR_T(x)$ if $f$ satisfies the following clauses: 
\begin{enumerate}
    \item $f(\bot) \equiv 0 = 1$, 
    \item $f(A \circ B) \equiv (f(A) \circ f(B))$ for $\circ \in \{\land, \lor, \to\}$, 
    \item $f(\neg A) \equiv \neg f(A)$, 
    \item $f(\Box A) \equiv \PR_T(\gn{f(A)})$. 
\end{enumerate}
The \emph{provability logic} of $\PR_T(x)$ is the set of all modal formulas $A$ such that $T \vdash f(A)$ for all arithmetical interpretations $f$ based on $\PR_T(x)$. 
A logic $L$ is \emph{arithmetically complete} if there exists a provability predicate $\PR_T(x)$ such that $L$ is the provability logic of $\PR_T(x)$. 
See Boolos \cite{Boolos}, Japaridze and de Jongh \cite{JD}, and Artemov and Beklemishev \cite{AB} for the details of provability logic. 

\paragraph{(1) The pure logic of necessitation $\N$ and its extensions.}
The pure logic of necessitation $\N$ is obtained by adding the necessitation rule \textsc{Nec}: $\dfrac{A}{\Box A}$ to classical propositional logic in the language of modal propositional logic.
The rule \textsc{Nec} corresponds to the derivability condition $\D{1}$, which is a property common to all provability predicates.
The logic $\N$ was introduced by Fitting, Marek, and Truszczy\'nski~\cite{FMT}, and they also developed a Kripke-like relational semantics for $\mathbf{N}$. 
The second author \cite{Kura24} proved that $\N$ is exactly the provability logic of all provability predicates by using the finite frame property of $\N$ with respect to their relational semantics.
The modal axiom $\mathsf{4}$: $\Box A \to \Box\Box A$ and the rule \textsc{Ros}: $\dfrac{\neg A}{\neg \Box A}$ correspond respectively to $\D{3}$ and $\mathbf{Ros}$.
The arithmetical completeness of the extensions $\mathsf{N4}$, $\mathsf{NR}$, and $\mathsf{NR4}$ obtained by adding these principles to $\N$ was established in \cite{Kura24}.

\paragraph{(2) Variations of $\mathsf{N4}$.}
The derivability condition $\D{3}^n_m$ is a generalization of $\D{3}$, and its modal counterpart is the Geach-type axiom $\mathsf{Acc}_{m,n}: \Box^n A \to \Box^m A$. 
For every $m, n \geq 1$, the finite frame property of the logic $\mathsf{NA}_{m,n}$ obtained by adding $\mathsf{Acc}_{m,n}$ to $\N$ was established by the second author and Sato \cite{KS}. 
The provability logical analysis of the logic $\mathsf{NA}_{m,n}$ was developed by the first author \cite{Kogure1}.
In particular, the arithmetical completeness of $\mathsf{NA}_{m,n}$ for all $m,n \ge 1$ was proved.

\paragraph{(3) On the rule $\mathbf{E}$.}
The first author \cite{Kogure2} investigated extensions of the logic~$\mathsf{E}$, which is obtained from classical propositional logic by adding the rule \textsc{RE}: $\dfrac{A \leftrightarrow B}{\Box A \leftrightarrow \Box B}$ corresponding to the derivability condition $\mathbf{E}$.
For these logics, the \emph{neighborhood semantics} (see~\cite{Chellas, Pacuit}) is well developed as the basic semantics.
Therefore, in order to prove the arithmetical completeness for these systems, one cannot directly apply the standard techniques used in the proof of Solovay's arithmetical completeness theorem, which rely on relational frames.
The first author developed a method of embedding neighborhood models into arithmetic and actually proved the arithmetical completeness of the logic $\mathsf{EN}$.
Furthermore, the modal principles corresponding respectively to $\ConL_T$ and $\ConS_T$ are $\mathsf{P}$: $\neg \Box \bot$ and $\mathsf{D}$: $\neg(\Box A \land \Box \neg A)$.
The first author also established the arithmetical completeness of the extensions $\mathsf{ENP}$, $\mathsf{ECN}$, and $\mathsf{ECND}$ of $\mathsf{EN}$. 
It should be noted that under the rule \textsc{RE}, the axiom $\mathsf{P}$ and the rule \textsc{Ros} are equivalent, and under the axiom $\mathsf{C}: (\Box p \land \Box q) \to \Box (p \land q)$, the rule \textsc{Ros} and the axiom $\mathsf{D}$ are equivalent.
However, the arithmetical completeness of the logics $\mathsf{EN4}$, 
$\mathsf{ENP4}$, and $\mathsf{ECN4}$, which include the axiom $\mathsf{4}$, 
remains open, since difficulties arise when embedding the corresponding 
neighborhood models into arithmetic.
The usual Solovay-style arithmetical completeness argument is essentially
based on relational semantics, where the behavior of $\Box$ is controlled by
accessibility relations. 
By contrast, logics based on the rule \textsc{RE} are naturally treated by
neighborhood semantics, and the axiom $\mathsf{4}$ does not simply correspond to
transitivity of a binary relation in that setting. 
This is one source of difficulty in proving arithmetical completeness for the
\textsc{RE}-based logics with $\mathsf{4}$.
Theorem \ref{Kura}.3 is a version of G2 stating that there is no provability predicate satisfying $\mathbf{E}$ and $\D{3}$ such that $T \vdash \ConS_T$, but this fact does not immediately imply that $\mathsf{END4}$ is not arithmetically complete.

\paragraph{(4) On the rule $\mathbf{M}$.}
The logic $\mathsf{MN}$ is obtained from $\N$ by adding the rule \textsc{RM}: $\dfrac{A \to B}{\Box A \to \Box B}$. 
The provability logical study of the logic $\mathsf{MN}$ and its extensions was developed by the authors \cite{KK}.
Fortunately, the neighborhood frames for these logics can be transformed into suitable Kripke-like relational frames, and this relational semantics enables one to establish their arithmetical completeness by applying Solovay's technique.
In fact, the arithmetical completeness of the logics $\mathsf{MN}$, $\mathsf{MNP}$, $\mathsf{MND}$, $\mathsf{MN4}$, and $\mathsf{MNP4}$ was obtained in \cite{KK}.
As in the case of $\mathbf{E}$, there is no provability predicate satisfying $\mathbf{M}$ and $\D{3}$ such that $T \vdash \ConS_T$, but this fact may not imply that $\mathsf{MND4}$ is not arithmetically complete.

\paragraph{(5) Normal modal logics.}
It is well known that the combination of the rules \textsc{Nec} and \textsc{RM} and the axiom $\mathsf{C}$ yields the axiom $\Box(A \to B) \to (\Box A \to \Box B)$.
So, by adding the axiom $\mathsf{C}$ to the logics $\mathsf{MN}$ and $\mathsf{MND}$, one obtains the normal modal logics $\mathsf{K}$ and $\mathsf{KD}$, respectively.
For the modal logic~$\mathsf{K}$, the second author \cite{Kura18} proved its arithmetical completeness with respect to Fefermanian $\Sigma_2$ provability predicates introduced in \cite{Fef}, and later the second author \cite{Kura24} established its arithmetical completeness with respect to $\Sigma_1$ provability predicates.
For~$\mathsf{KD}$, the existence of a Fefermanian $\Sigma_2$ provability predicate whose $\ConL_T$ is $T$-provable was proved by Feferman \cite{Fef}. 
Later, the existence of Rosser provability predicates satisfying $\D{2}$ was proved by Bernardi and Montagna \cite{BM84} and Arai \cite{Arai} (see also \cite{Viss89,Shav94}). 
The arithmetical completeness of $\mathbf{KD}$ with respect to Rosser provability predicates satisfying~$\D{2}$ was proved by the second author \cite{Kura20b}.

On the other hand, the modal logic~$\mathsf{K4}$ corresponds to the derivability conditions~$\D{1}$–$\D{3}$, and for any provability predicate $\PR_T(x)$ satisfying these conditions, the formalized L\"ob's theorem ensures that the provability logic of $\PR_T(x)$ contains $\GL$.
Hence, $\mathsf{K4}$ is not arithmetically complete.
Montague~\cite{Mont} showed that any logic containing~$\mathsf{KT}$ is not arithmetically complete.
Furthermore, it is proved that any logic containing $\mathsf{KD4} \cap \mathsf{KD5} \cap \mathsf{KT}$ is not arithmetically complete \cite{Kura20b}, and that among the logics containing $\mathsf{KB} \cap \mathsf{K5}$, the only arithmetically complete logic is $\mathsf{Ver} := \mathsf{K} + \Box \bot$ \cite{Kura18b}.

Similarly, for each $n \geq 1$, Sacchetti \cite{Sacc} showed that the provability logic of any provability predicate satisfying $\D{2}$ and $\D{3}^1_{n+1}$ contains the logic $\mathsf{K} + (\Box(\Box^n A \to A) \to \Box A)$.
In this sense, the logic $\mathsf{K} + (\Box A \to \Box^{n+1} A)$ cannot be a provability logic.
On the other hand, the second author \cite{Kura18b} proved that $\mathsf{K} + (\Box(\Box^n A \to A) \to \Box A)$ is arithmetically complete with respect to Fefermanian $\Sigma_2$ provability predicates.

\subsection{Remaining open problems and the goal of this paper}

In the above studies, the logics whose arithmetical completeness has not yet been established can be divided into three groups.
The first group consists of $\mathsf{EN4}$, $\mathsf{ENP4}$, and $\mathsf{ECN4}$, as mentioned above.
The difficulty of proving the arithmetical completeness for these logics indicates that our research is not merely a matter of routine extension, but involves technically and conceptually challenging aspects. 
Indeed, it is also possible that logics such as $\mathsf{EN4}$ are not arithmetically complete.
In that case, a new principle in arithmetic may be derived from $\mathbf{E}$ and $\D{3}$ by applying the Fixed-Point Theorem (cf.~\cite{Lind}). 
Such a situation would be interesting.
The existence of a $\Sigma_1$ provability predicate satisfying $\mathbf{E}$, $\mathbf{C}$, and $\D{3}$, but not $\mathbf{M}$ was proved by the second author \cite{Kura25}, which corresponds to the logic $\mathsf{ECN4}$. 

The second group consists of $\mathsf{CN}$, $\mathsf{CNP}$, $\mathsf{CND}$, $\mathsf{CN4}$, and $\mathsf{CNP4}$.
The modal completeness with respect to the semantics developed by Fitting, Marek, and Truszczy\'{n}ski for logics including the axiom $\mathsf{C}$ has not yet been established, and therefore the arithmetical completeness for these logics remains open. 
It follows from Theorem \ref{Kura}.4 that the logic $\mathsf{CND4}$ is not arithmetically complete. 
It is easy to see that Mostowski's provability predicate \cite[p.~24]{Mos65} satisfies $\mathbf{C}$, $\D{3}$, and $T \vdash \ConL_T$, and so it corresponds to the logic $\mathsf{CNP4}$.

The final group, and the main focus of the present paper, consists of $\NP$, $\ND$, $\NPF$, and $\NDF$.
The present paper is a continuation of the line of research reviewed in this section, and its aim is to settle the arithmetical completeness of these logics.
To achieve this, two main difficulties must be overcome.
The first is to establish the modal completeness and the finite frame property of these logics with respect to the semantics of Fitting, Marek, and Truszczy\'{n}ski.
This requires special care not discussed in the previous related works \cite{FMT, KS}, and the details are presented in Section \ref{MC}.
The second is to prove the arithmetical completeness of these logics themselves.
In particular, the case of $\NDF$ requires an essential new idea.
Arai~\cite{Arai} showed the existence of a Rosser provability predicate $\PRR_T(x)$ whose provability logic includes $\NDF$, but unfortunately, it is easily seen that the provability logic of his $\PRR_T(x)$ strictly includes $\NDF$ (for example, it contains $\Box \neg\neg p \leftrightarrow \Box p$, which is not provable in $\NDF$).
Hence, we proved the arithmetical completeness of $\NDF$ by applying Arai's construction method only to formulas lying outside the arithmetical interpretation.
We prove the arithmetical completeness theorem for $\ND$ and $\NDF$ in Section \ref{AC1}. 
Section \ref{AC2} is devoted to proving the arithmetical completeness of $\NP$ and $\NPF$. 

The current situation is summarized in Table \ref{Table1}. 
In the table, `AC' stands for `Arithmetical Completeness'. 
Also, `$\exists$' stands for the existence of a provability predicate satisfying exactly the pattern of conditions specified in that row. 
In the columns for the derivability conditions and consistency principles, + indicates that the corresponding condition or principle is satisfied, and - indicates that it is not. 
For the columns `AC' and `$\exists$', the symbol - means, respectively, that the corresponding logic is not arithmetically complete and that there is no provability predicate satisfying the specified pattern of conditions.

\begin{table}[ht]
\centering
{\scriptsize
\begin{tabular}{|c|c|c||c|c|c|c|c|c|c|}
\hline
Logic & AC & $\exists$ & $\mathbf{E}$ & $\mathbf{M}$ & $\mathbf{C}$ & $\mathbf{D3}$ & $\vdash \ConL_T$ & $\mathbf{Ros}$ & $\vdash \ConS_T$ \\
\hline
\hline
$\N$ & \cite{Kura24}  &  \cite{Kura24} &  - & - & - & - & - & - & - \\
\hline
$\NP$ & Cor.~\ref{AC_NP}  & $\checkmark$ & - & - & - & - & + & - & - \\
\hline
$\mathsf{NR}$ & \cite{Kura24}  & \cite{Shavrukov,Kura24} & - & - & - & - & + & + & - \\
\hline
$\ND$ & Cor.~\ref{AC_ND}  & $\checkmark$ & - & - & - & - & + & + & + \\
\hline
$\mathsf{N4}$ & \cite{Kura24,Kogure1}  & \cite{Kura24,Kogure1} & - & - & - & + & - & - & - \\
\hline
$\NPF$ & Cor.~\ref{AC_NP4}  & $\checkmark$ & - & - & - & + & + & - & - \\
\hline
$\mathsf{NR4}$ & \cite{Kura24}  & \cite{Kura24} & - & - & - & + & + & + & - \\
\hline
$\NDF$ & Cor.~\ref{AC_ND4}  & \cite{Arai} & - & - & - & + & + & + & + \\
\hline
$\mathsf{CN}$  &  &  & - & - & + & - & - & - & - \\
\hline
$\mathsf{CNP}$  &  &  & - & - & + & - & + & - & - \\
\hline
$\mathsf{CND}$  &  &  & - & - & + & - & + & + & + \\
\hline
$\mathsf{CN4}$  &  &  & - & - & + & + & - & - & - \\
\hline
$\mathsf{CNP4}$  &  & \cite{Mos65} & - & - & + & + & + & - & - \\
\hline
$\mathsf{CND4}$ & - & - \cite{Kura25} & - & - & + & + & + & + & + \\
\hline
$\mathsf{EN}$ & \cite{Kogure2}  & \cite{Kogure2} & + & - & - & - & - & - & - \\
\hline
$\mathsf{ENP}$ & \cite{Kogure2}  & \cite{Kogure2} & + & - & - & - & + & + & - \\
\hline
$\mathsf{END}$ & \cite{Kogure2}  & \cite{Kogure2} & + & - & - & - & + & + & + \\
\hline
$\mathsf{EN4}$ &   &  & + & - & - & + & - & - & - \\
\hline
$\mathsf{ENP4}$ &   &  & + & - & - & + & + & + & - \\
\hline
$\mathsf{END4}$ &   & - \cite{Kura25} & + & - & - & + & + & + & + \\
\hline
$\mathsf{ECN}$ & \cite{Kogure2}  & \cite{Kogure2}  & + & - & + & - & - & - & - \\
\hline
$\mathsf{ECND}$ & \cite{Kogure2}  & \cite{Kogure2}  & + & - & + & - & + & + & + \\
\hline
$\mathsf{ECN4}$  &  & \cite{Kura25} & + & - & + & + & - & - & - \\
\hline
$\mathsf{MN}$ & \cite{KK}  & \cite{KK} & + & + & - & - & - & - & - \\
\hline
$\mathsf{MNP}$ & \cite{KK}  & \cite{KK} & + & + & - & - & + & + & - \\
\hline
$\mathsf{MND}$ & \cite{KK}  & \cite{KK} & + & + & - & - & + & + & + \\
\hline
$\mathsf{MN4}$ & \cite{KK}  & \cite{KK} & + & + & - & + & - & - & - \\
\hline
$\mathsf{MNP4}$ & \cite{KK}  & \cite{Kura21, KK} & + & + & - & + & + & + & - \\
\hline
$\mathsf{MND4}$ &  & - \cite{Kura20} & + & + & - & + & + & + & + \\
\hline
$\mathsf{K}$ & \cite{Kura18,Kura24}  & \cite{Kura18,Kura24} & + & + & + & - & - & - & - \\
\hline
$\mathsf{KD}$ & \cite{Kura20b}  & \cite{Fef,BM84,Arai,Kura20b} & + & + & + & - & + & + & + \\
\hline
\end{tabular}
}
\caption{The summary of the current situation}\label{Table1}
\end{table}

\section{Modal completeness}\label{MC}

Let $\mathsf{Prop}$ be a countably infinite set of propositional variables. The set $\mathsf{MF}$ of all modal formulas is defined by the following grammar:
\[
A ::= p \mid \bot \mid (A \land A) \mid (A \lor A) \mid (A \to A)
      \mid \neg A \mid \Box A,
\]
where $p \in \mathsf{Prop}$.
Let $\top$ be an abbreviation for $\neg\bot$.

Fitting, Marek, and Truszczy\'{n}ski \cite{FMT} introduced the pure logic of necessitation $\N$,
which is obtained by adding the necessitation rule \textsc{Nec}: $\dfrac{A}{\Box A}$ to classical propositional logic in this language. 
The logic $\N$ is not normal because the distribution axiom $\Box (p \to q) \to (\Box p \to \Box q)$ is not provable in $\N$.
They introduced the following relational semantics for $\N$.

\begin{defn}[$\N$-frames and $\N$-models]
\leavevmode

\begin{itemize}
\item 
A tuple $(W, \{ \prec_B \}_{B \in \MF} )$ is called an $\N$-\textit{frame}
if $W$ is a non-empty set and $\prec_B$ is a binary relation on $W$ for every $B \in \MF$. 
\item 
An $\N$-frame $(W, \{ \prec_B \}_{B \in \MF} )$ is said to be finite if $W$ is finite.
\item 
A triple $(W, \{ \prec_B \}_{B \in \MF}, \Vdash)$ is called an $\N$-\textit{model}
if $(W, \{ \prec_B \}_{B \in \MF} )$ is an $\N$-frame and $\Vdash$ is a binary relation
between $W$ and $\MF$ satisfying the usual conditions for satisfaction and the following condition:
\[
    x \Vdash \Box B \iff \forall y \in W \, (x \prec_B y \Longrightarrow y \Vdash B).
\]
\item 
A formula $A$ is \textit{valid} in an $\N$-model $(W, \{ \prec_B \}_{B \in \MF}, \Vdash)$
if $x \Vdash A$ for all $x \in W$.
\item 
A formula $A$ is \textit{valid} in an $\N$-frame $ \mathcal{F} = (W, \{ \prec_B \}_{B \in \MF})$
if $A$ is valid in all $\N$-models $(\mathcal{F}, \Vdash)$ based on $\mathcal{F}$.
\end{itemize}
\end{defn}

Fitting, Marek, and Truszczy\'{n}ski proved that
$\N$ is sound and complete with respect to their semantics.
In addition, $\N$ has the finite frame property with respect to that semantics.

\begin{thm}[{\cite[Theorems 3.6 and 4.10]{FMT}}]
$\N$ is sound and complete with respect to the class of all $\N$-frames. 
Moreover, $\N$ has the finite frame property. 
\end{thm}

We introduce the following seven logics that are extensions of $\N$. 
The logics $\mathsf{N4}$, $\mathsf{NR}$, and $\mathsf{NR4}$ were introduced in \cite{Kura24}. 
The logics $\NP$, $\ND$, $\NPF$, and $\NDF$ are the main focus of our present paper. 

\begin{defn}
\leavevmode
\begin{itemize}
    \item $\NP : = \N + \neg \Box \bot$.
    \item $\mathsf{NR} : = \N + \dfrac{\neg A}{\neg \Box A}$.
    \item $\ND : = \N + \neg (\Box A \wedge \Box \neg A)$.
    \item $\mathsf{N4} : = \N + (\Box A \to \Box \Box A)$.
    \item $\NPF : = \NP + (\Box A \to \Box \Box A)$.
    \item $\mathsf{NR4} : = \mathsf{NR} + (\Box A \to \Box \Box A)$.
    \item $\NDF : = \ND + (\Box A \to \Box \Box A)$. 
\end{itemize}
\end{defn}

It is easy to see that $\NP \subseteq \mathsf{NR} \subseteq \ND$. 
Although $\neg \Box \bot$, $\dfrac{\neg A}{\neg \Box A}$, and $\neg (\Box A \land \Box \neg A)$ are equivalent over normal modal logics, they are distinguished over the weak logic $\N$. 
Semantical analysis of the logics $\mathsf{N4}$, $\mathsf{NR}$, and $\mathsf{NR4}$ has already been developed in \cite{Kura24}.

\begin{defn}
Let $\mathcal{F} = (W, \{ \prec_B\}_{B \in \MF})$ be an $\N$-frame and $\Gamma \subseteq \MF$.
\begin{itemize}
    \item 
    $\mathcal{F}$ is said to be $\Gamma$-transitive if for any $\Box \Box A \in \Gamma$ and $x,y, z \in W$, if $x \prec_{\Box A} y$ and $y \prec_{A}z$, then $x \prec_{A} z$.
    \item 
    We say that $\mathcal{F}$ is transitive if it is $\MF$-transitive.

    \item 
    $\mathcal{F}$ is said to be serial if for any $A \in \MF$ and $x \in W$, there exists $y \in W$ such that $x \prec_A y$. 
\end{itemize}
\end{defn}

\begin{prop}[cf.~{\cite[Proposition 3.6]{Kura24}}]\label{propN0}
For any $A \in \MF$, if $\neg A$ is valid in a serial $\N$-frame $\mathcal{F}$, then $\neg \Box A$ is also valid in $\mathcal{F}$.
\end{prop}

\begin{prop}[cf.~{\cite[Proposition 3.9]{Kura24}}]\label{propN1}
For any $A \in \MF$, the formula $\Box A \to \Box \Box A$ is valid in all transitive $\N$-frames.
\end{prop}

\begin{thm}[{\cite[Theorems 3.12, 3.13, and 3.14]{Kura24}}]
Each of the logics $\mathsf{N4}$, $\mathsf{NR}$, and $\mathsf{NR4}$ is sound and complete with respect to the corresponding class of $\N$-frames. 
Moreover, each of them has the finite frame property. 
\end{thm}

The goal of this section is to prove the finite frame property of the four logics $\NP$, $\ND$, $\NPF$, and $\NDF$ with respect to corresponding classes of $\N$-frames. 

\begin{defn}
Let $\mathcal{F} = (W, \{ \prec_B\}_{B \in \MF})$ be an $\N$-frame.
\begin{itemize}
    \item 
    $\mathcal{F}$ is called an $\NP$-frame if
    for any $x \in W$, there exists $y \in W$ such that $x \prec_{\bot} y$.
    \item 
    We say that $\mathcal{F}$ is an $\ND$-frame if
    for any $A \in \MF$ and $x \in W$, there exists $y \in W$ such that $x \prec_{A} y$ and $x \prec_{\neg A} y$.
\end{itemize}
\end{defn}
We say that an $\N$-frame is an $\NPF$-frame if it is a transitive $\NP$-frame. 
Similarly, an $\N$-frame is called an $\NDF$-frame if it is a transitive $\ND$-frame.

\begin{prop}\label{propN2}
\leavevmode
\begin{enumerate}
    \item 
    $\neg \Box \bot$ is valid in all  $\NP$-frames.
\item 
    For any $A \in \MF$,
    $\neg (\Box A \wedge \Box \neg A)$ is valid in all $\ND$-frames.
\end{enumerate}
\end{prop}
\begin{proof}
1. Let $(W, \{ \prec_B\}_{B \in \MF}, \Vdash)$ be an $\N$-model based on an $\NP$-frame.
    Suppose, towards a contradiction, that $x \Vdash \Box \bot$ for some $x \in W$. Then, there exists $y \in W$ such that $x \prec_{\bot}y$, which implies $y \Vdash \bot$, a contradiction. Thus, we obtain $x \Vdash \neg \Box \bot$ for all $x \in W$.

\medskip

2. Let $(W, \{ \prec_B\}_{B \in \MF}, \Vdash)$ be an $\N$-model based on an $\ND$-frame.
   Suppose, towards a contradiction, that $x \Vdash \Box A \wedge \Box \neg A$ for some $x \in W$ and $A \in \MF$.
Then, there exists $y \in W$ such that $x \prec_{A} y$ and $x \prec_{\neg A} y$, which implies $y \Vdash A$ and $y \Vdash \neg A$, a contradiction.
Therefore, we obtain $x \Vdash \neg (\Box A \wedge \Box \neg A)$ for all $A \in \MF$ and $x \in W$.
\end{proof}

We shall prove the finite frame property of the logics $\ND$, $\NP$, $\NDF$, and $\NPF$.
For each $A \in \MF$,  if $A$ is of the form $\neg B$, then  let ${\sim} A$ be $B$; otherwise, let ${\sim} A$ be $\neg A$. For each $k \in \omega$, we define $\neg^k A$ inductively by $\neg^0 A : \equiv A$ and $\neg^{k+1} A : \equiv \neg \neg^k A$.
Let $\Sub(A)$ be the set of all subformulas of $A$ and $\Sub^* (A)$ be the set
\[
\Sub(A) \cup \{ \Box  B \in \MF \mid \Box \neg^kB \in \Sub(A) \text{ for some } k \geq 0 \}. 
\]
It is easily shown that if $\Box \neg^k B \in \Sub(A)$, then $\Box \neg^i B \in \Sub^\ast(A)$ for all $i \leq k$. 
So, if $\Box \neg B \in \Sub^\ast(A)$, then $\Box B \in \Sub^\ast(A)$. 

Let $\num{\Sub(A)}$ be the union of the following three sets: 
\begin{enumerate}
    \item $\Sub^*(A)$, 
    \item $\{ {\sim} {B} \mid B \in \Sub^*(A) \}$, 
    \item $\{ \Box \bot, \neg \Box \bot, \Box \top, \neg \Box \top, \bot, \top \}$.
\end{enumerate}
Notice that $\num{\Sub(A)}$ is a finite set of formulas. 
It is shown that the set $\num{\Sub(A)}$ is closed under taking subformulas. 

Let $L$ be one of $\ND$, $\NP, \NDF$, and $\NPF$.
We say that a finite set $X \subseteq \MF$ is $L$-consistent if $L \nvdash \bigwedge X \to \bot$, where $\bigwedge X$ denotes a conjunction of all elements of $X$. 
We say that $X$ is $A$-maximally $L$-consistent if $X \subseteq \num{\Sub(A)}$, $X$ is $L$-consistent, and for any $B \in \num{\Sub(A)}$, either $B \in X$ or ${\sim} B \in X$.
It is easily shown that for each $L$-consistent subset $X$ of $\num{\Sub(A)}$,
there exists an $A$-maximally $L$-consistent superset of $X$. 

We are ready to prove our modal completeness theorem. 

\begin{thm}\label{thm:MC}
Let $L \in \{ \NP, \ND, \NPF,  \NDF \}$. 
Then, for any $A \in \MF$, the following are equivalent: 
\begin{enumerate}
    \item $L \vdash A$. 
    \item $A$ is valid in all $L$-frames.
    \item $A$ is valid in all finite $L$-frames.
\end{enumerate}
\end{thm}
\begin{proof}
The implications $(1 \Rightarrow 2)$ and $(2 \Rightarrow 3)$ are straightforward by Proposition \ref{propN2}.
We prove the contrapositive of the implication $(3 \Rightarrow 1)$.
Suppose $L \nvdash A$. 
Then the set $\{ {\sim}A \}$ is $L$-consistent, and there exists an $A$-maximally consistent set $x_A$ such that $\{ {\sim} A \} \subseteq x_A$.

We define the triple $(W, \{\prec_B\}_{B \in \MF}, \Vdash)$ as follows.
\begin{itemize}
    \item $W:= \{ x \subseteq \num{\Sub(A)} \mid x \text{ is an } A \text{-maximally } L\text{-consistent set} \}$.
    \item $x \prec_B y :\iff \Box B \notin x $ or $B \in y$.
    \item $x \Vdash p : \iff p \in x$.
\end{itemize}
Since $\num{\Sub(A)}$ is finite and $W$ contains $x_A$, $(W, \{ \prec_{B}\}_{B \in \MF})$ is a finite $\N$-frame.
The following claim is proved in the same way as in \cite[Claim 3.17]{Kura24}.

\begin{cl}\label{ffp-1}
For any $x \in W$  and $B \in \num{\Sub(A)}$,
\[
x \Vdash B \iff B \in x.
\]
\end{cl}
Since $A \notin x_A$, we obtain $x_A \nVdash A$ by Claim \ref{ffp-1}.
So, $A$ is not valid in the $\N$-model $(W, \{\prec_B\}_{B \in \MF}, \Vdash)$. 
The following claim completes the proof of the finite frame property for the case $L = \NP$.

\begin{cl}\label{ffp-2}
If $L \in \{ \NP, \NPF \}$, then $(W, \{\prec_B\}_{B \in \MF})$ is an $\NP$-frame.
\end{cl}
\begin{proof}
Let $x \in W$.
Since $L \vdash \neg \Box \bot$ and $x$ is $L$-consistent,  it follows that $\Box \bot \notin x$.
Therefore, we obtain $x \prec_{\bot} x$.
\end{proof}

Similarly, the following claim completes the proof for the case $L = \ND$.

\begin{cl}\label{ffp-3}
If $L \in \{ \ND, \NDF \}$, then $(W, \{\prec_B\}_{B \in \MF})$ is an $\ND$-frame.
\end{cl}
\begin{proof}
Let $C \in \MF$ and $x \in W$.
Since $\ND \vdash \Box C \wedge \Box \neg C \to \bot$, it follows from the $L$-consistency of $x$ that at least one of $\Box C$ and $\Box \neg C$ is not in $x$. 
We distinguish the following three cases.

\medskip

Case 1. $\Box C \notin x$ and $\Box \neg C \notin x$:
In this case, by the definition of $\{\prec_B \}_{B \in \MF}$, we obtain $x \prec_C x$ and $x \prec_{\neg C} x$.

\medskip

Case 2. $\Box C \notin x$ and $\Box \neg C \in x$:
Since $\Box \neg C \in \num{\Sub(A)}$, we obtain $\neg C \in \num{\Sub(A)}$.
Suppose, towards a contradiction, that the set $\{ \neg C \}$ is $L$-inconsistent.
Then, we obtain $L \vdash C$, which implies $L \vdash \Box C$. Since $L \vdash \Box C \wedge \Box \neg C \to \bot$, it follows that $L \vdash \Box \neg C \to \bot$.
This contradicts the $L$-consistency of $x$.
Therefore, the set $\{ \neg C \}$ is $L$-consistent, and so there exists $y \in W$ such that $\{ \neg C \} \subseteq y$. 
Thus, by the $L$-consistency of $y$, we obtain $C \notin y$. 
Since $\neg C \in \num{\Sub(A)}$, we have $\neg C \in y$, which implies $x \prec_{\neg C} y$.
Since $\Box C \notin x$, it follows that $x \prec_{C}y$.

\medskip 

Case 3. $\Box C \in x$ and $\Box \neg C \notin x$:
In this case, the existence of $z \in W$ such that $x \prec_C z$ and $x \prec_{\neg C} z$ is proved analogously as in Case 2.
\end{proof}

In the case where $L \in \{\NPF, \NDF\}$, the frame $(W, \{\prec_B\}_{B \in \MF})$ is not necessarily transitive in general, but the following holds.

\begin{cl}\label{ffp-4}
    If $L \in \{ \NDF, \NPF\}$, then $(W, \{\prec_B\}_{B \in \MF})$ is $\Sub^*(A)$-transitive.
\end{cl}
\begin{proof}
Let $\Box \Box C \in \Sub^*(A)$.
Suppose $x \prec_{\Box C} y$ and $y \prec_{C} z$.
We distinguish two cases according to whether $\Box \Box C \in x$ or not.
If $\Box \Box C \in x$, then it follows from $x \prec_{\Box C} y$ that $\Box C \in y$.
Since $y \prec_{C} z$, we obtain $C \in z$.
Therefore, $x \prec_{C} z$ holds.
If $\Box \Box C \notin x$, then we obtain $\neg \Box \Box C \in x$ because $\neg \Box \Box C \in \num{\Sub(A)}$.
Since $L \vdash \neg \Box \Box C \to \neg \Box C$, we obtain $\neg \Box C \in x$, that is, $\Box C \notin x$.
Thus, it follows that $x \prec_C z$.
\end{proof}

Therefore, when $L \in \{\NPF, \NDF\}$, we have to reconstruct the frame $(W, \{\prec_B\}_{B \in \MF})$ to define a new transitive $\N$-frame.

This reconstruction method has been developed in \cite{Kura24, KS}, and in particular, when $L = \NPF$, a transitive $\NP$-frame can be obtained by the existing method. 
However, in the case of $L = \NDF$, the existing method does not preserve the property of being an $\ND$-frame.
Therefore, we introduce a new construction as follows.
Since this construction also works for $L = \NPF$, from now on, we proceed with the proof for $L \in \{\NPF, \NDF\}$.

Let $=^W$ denote the equality $\{(w, w) \mid w \in W\}$ on $W$. 
We define the finite $\N$-model $(W, \{ \prec^*_B\}_{B \in \MF}, \Vdash^*)$ as follows.

\begin{itemize}
    \item $\prec^*_{B} := \begin{cases}
     \prec_B & \text{if } \Box B \in \Sub^\ast(A) \text{ or}\\
     & \quad B \equiv \neg^k C \text{ for some } k >0 \text{ and } \Box C \in \Sub(A), \\
     =^W & \text{otherwise.}
    \end{cases}$
    \item $x \Vdash^* p : \iff x \Vdash p$.
\end{itemize}

Since we have $\prec_C = \prec_C^\ast$ for every $\Box C \in \Sub(A)$, the following claim is easily verified. 

\begin{cl}\label{ffp'}
For any $B \in \Sub(A)$ and $x \in W$,
\[
x \Vdash B \iff x \Vdash^* B.
\]
\end{cl}

In particular, we have $x_A \nVdash^\ast A$ because $A \in \Sub(A)$. 
So, it suffices to prove that our new $\N$-frame $(W, \{ \prec^*_B\}_{B \in \MF})$ is an $L$-frame.

\begin{cl}\label{ffp-5}
If $L = \NPF$, then $(W, \{\prec^*_B\}_{B \in \MF})$ is an $\NP$-frame.
\end{cl}
\begin{proof}
By the definition of $\prec_\bot^\ast$, we have that $\prec_\bot^\ast$ is either $\prec_\bot$ or $=^W$. 
Claim \ref{ffp-2} guarantees that for every $x \in W$, there exists $y \in W$ such that $x \prec_\bot^\ast y$. 
\end{proof}

\begin{cl}\label{ffp-6}
 If $L = \NDF$, then $(W, \{\prec^*_B\}_{B \in \MF})$ is an $\ND$-frame. 
\end{cl}
\begin{proof}
Let $C \in \MF$ and $x \in W$. 
By Claim \ref{ffp-3}, there exists $y \in W$ such that $x \prec_{C} y$ and $x \prec_{\neg C} y$.
We distinguish the following three cases.

\paragraph{Case 1.} $\Box C \in \Sub^*(A)$: 
In this case, $\prec^*_{C} = \prec_{C}$. 
By the definition of $\Sub^\ast(A)$, we find some $k \geq 0$ such that $\Box \neg^k C \in \Sub(A)$. 
If $k = 0$, then $\Box C \in \Sub(A)$, and hence we have $\prec^*_{\neg C} = \prec_{\neg C}$. 
If $k > 0$, then $\Box \neg C \in \Sub^\ast(A)$, and thus we also have $\prec^*_{\neg C} = \prec_{\neg C}$. 
In either case, we get $x \prec^*_{C} y$ and $x \prec^*_{\neg C}y$.

\paragraph{Case 2.} $C \equiv \neg^k D$ for some $k > 0$ and $\Box D \in \Sub(A)$:
We have $\prec^*_{C} = \prec_{C}$.
    Since $\neg C  \equiv \neg^{k+1} D$, we also have $\prec^*_{\neg C} = \prec_{\neg C}$.

\paragraph{Case 3.} Otherwise: 
We have that $\prec_{ C}^\ast$ is $=^W$. 
We prove that $ \prec_{\neg C}^\ast$ is also $=^W$.
Since $\Box C \notin \Sub^\ast(A)$, we have $\Box \neg C \notin \Sub^\ast(A)$. 
Suppose, towards a contradiction, that $\neg C \equiv \neg^{k} D$ for some $k>0$ and $\Box D \in \Sub(A)$.
Then, $C \equiv \neg^{k-1} D$.
Since $\Box C \notin \Sub^*(A)$, we have $k-1 = 0$, and so $k=1$.
It follows that $C \equiv D$, and hence $\Box C \in \Sub(A)$, a contradiction.
Thus, we obtain that $ \prec_{\neg C}^\ast$ is $=^W$. 
We conclude that $x \prec_C^\ast x$ and $x \prec_{\neg C}^\ast x$. 
\end{proof}

Finally, we prove the following claim. 

\begin{cl}\label{ffp-7}
The $\N$-frame $(W, \{\prec^*_B\}_{B \in \MF})$ is transitive.
\end{cl}
\begin{proof}
Suppose $x \prec^*_{\Box C}y$ and $y \prec^*_{C} z$ for $C \in \MF$. 
We distinguish the following two cases: 

\medskip

1. $\Box \Box C \in \Sub^*(A)$:
We find some $k \geq 0$ such that $\Box \neg^k \Box C \in \Sub(A)$. 
So, $\Box C \in \Sub(A) \subseteq \Sub^*(A)$. 
It follows that $\prec^*_{\Box C} = \prec_{\Box C}$ and $\prec^*_{C} = \prec_{ C}$.
Then, we have $x \prec_{\Box C} y$ and $x \prec_{C} z$. 
By Claim \ref{ffp-4}, $x \prec_C z$, and hence $x \prec_C^\ast z$. 

\medskip
    
2. $\Box \Box C \notin \Sub^*(A)$:
Since $\Box C$ is not of the form $\neg D$, we have that $\prec^\ast_{\Box C}$ is $=^W$. 
Since $x \prec^*_{\Box C} y$, we get $x=y$. 
Thus, it follows from $y \prec^*_{C}z$ that $x \prec^*_{C}z$.
\end{proof}

We have finished our proof of the implication $(3 \Rightarrow 1)$.
\end{proof}



For a given formula $A$, the finite set $\num{\Sub(A)}$ is primitive recursively computable from $A$. 
If $L \nvdash A$, the proof constructs a countermodel whose worlds are subsets of $\num{\Sub(A)}$, and hence whose number of worlds is bounded by $2^{|\num{\Sub(A)}|}$.  
In the cases $L=\NPF$ and $L=\NDF$, the reconstruction replacing $\prec_B$ by $\prec_B^\ast$ leaves the set of worlds unchanged.
Thus the proof yields a primitive recursive decision procedure: given $A$, one searches through finite models whose sets of worlds have size at most $2^{|\num{\Sub(A)}|}$, and checks the truth of $A$ by bounded quantification over the finite data relevant to $A$.
If such a countermodel is found, then $L \nvdash A$; otherwise, by the theorem, $L \vdash A$. 
Therefore we obtain the following corollary.

\begin{cor}\label{decidability}
For each $L \in \{ \ND, \NP, \NDF, \NPF \}$, the $L$-provability problem is primitive recursively decidable. 
Moreover, if $L \nvdash A$, then a finite $L$-model falsifying $A$ can be constructed primitive recursively from $A$.
\end{cor}

\section{The arithmetical completeness of $\ND$ and $\NDF$}\label{AC1}

In this section, we prove the arithmetical completeness theorem for $\ND$ and $\NDF$.
Before proving the theorem, we introduce several notions, which are used throughout the rest of this paper. 
Let $\langle \xi_t \rangle_{t \in \omega}$ be the repetition-free primitive recursive enumeration of all $\LA$-formulas in ascending order of G\"odel numbers. 
We call an $\LA$-formula \textit{propositionally atomic} if it is either atomic or of the form $Q x \psi$, where $Q \in \{\forall, \exists \}$ and $\psi$ is an arbitrary $\LA$-formula.
For each propositionally atomic formula $\varphi$, we prepare a distinct propositional variable $p_{\varphi}$.
We define a primitive recursive injection $I$ from the set of all $\LA$-formulas into a set of propositional formulas as follows:
\begin{itemize}
\item 
$I(\varphi)$ is $p_{\varphi}$ for each propositionally atomic formula $\varphi$,
\item 
$I(\varphi \circ \psi)$ is $I(\varphi) \circ I(\psi)$ for $\circ \in \{ \wedge, \vee, \to \}$, 
\item 
$I(\neg \varphi)$ is $\neg I(\varphi)$. 
\end{itemize}
Let $X$ be a finite set of $\LA$-formulas.
An $\LA$-formula $\varphi$ is called a \textit{tautological consequence} (\textit{t.c.}) of $X$
if $\bigwedge_{\psi \in X}I(\psi) \to I(\varphi)$ is a tautology.  
Let $X \tc \varphi$ denote that $\varphi$ is a t.c.~of $X$.
For each $n \in \omega$, let $F_n$ be the set of all $\LA$-formulas whose G\"odel number is less than or equal to $n$. 
We may assume that $F_0 = \emptyset$. 
Let
\[
    P_{T,n} : = \{ \varphi \mid \mathbb{N} \models \exists y \leq \num{n} \ \Proof_T(\gn{\varphi}, y)\},
\]
where $\Proof_T(x, y)$ is a standard primitive recursive proof predicate of $T$ naturally expressing that ``$y$ is the G\"{o}del number of a proof of $x$ from $T$''.
If $P_{T,n} \tc \varphi$, then $\varphi$ is provable in $T$.
Notice that $P_{T, n} \subseteq F_n$. 
The facts about the above notions can be formalized and verified in $\PA$.

We are ready to prove our main theorem of this section. 
The proof below is based on a method used in previous proofs of the arithmetical completeness of several logics \cite{Kogure2,Kogure1,KK,Kura20b,Kura24}. 
In this method, two primitive recursive functions are defined simultaneously.
One is a function $h$, which is used to select a world in one of the finite
countermodels, and the other is a function $g$, from which the desired
provability predicate is obtained. 
Unlike the ordinary Solovay function for $\GL$, the function $h$ changes
its value at most once, and after such a change it remains constant. 
The behavior of $g$ is arranged according to the final value of $h$ so
that the resulting provability predicate coincides with the desired modal logic.
The present construction, however, differs essentially from the cases treated in \cite{KK,Kura20b} because the rule \textsc{RM} is not available in the present setting. 

\begin{thm}\label{thmND4}
Let $L \in \{\ND, \NDF\}$. 
There exists a $\Sigma_1$ provability predicate $\PR_T(x)$ of $T$ satisfying the following properties: 
\begin{enumerate}
    \item \textup{(Arithmetical soundness)} 
    For any $A \in \MF$ and any arithmetical interpretation $f$ based on $\PR_T(x)$, if $L \vdash A$, then $\PA \vdash f(A)$; 
    \item \textup{(Uniform arithmetical completeness)} 
    There exists an arithmetical interpretation $f$ based on $\PR_T(x)$ such that for any $A \in \MF$, $L \vdash A$ if and only if $T \vdash f(A)$.
\end{enumerate}
\end{thm}

\begin{proof}
Let $L \in \{ \ND, \NDF \}$.
From Corollary \ref{decidability}, we find a primitive recursive enumeration $\langle A_k  \rangle_{k \in \omega}$ of all $L$-unprovable modal formulas.
For each $k \in \omega$, a finite $L$-model $\bigl( W_k, \{ \prec_{k,B}\}_{B \in \MF}, \Vdash_k \bigr)$ which falsifies $A_k$ is primitive recursively constructed.
We may assume that the sets $\{ W_k \}_{k \in \omega}$ are pairwise disjoint subsets of $\omega$ and $\bigcup_{k \in \omega} W_k = \omega \setminus \{0\}$. 
We may assume that each model $\bigl( W_k, \{ \prec_{k,B}\}_{B \in \MF}, \Vdash_k   \bigr)$ is primitive recursively represented in $\PA$ and that basic properties of the model are proved in $\PA$.

For each $\LA$-formula $\varphi$, we define an $\LA$-formula $\varphi^\star$ as follows: 
Let $\psi$ be a formula which is not of the form $\neg \chi$ such that $\varphi$ is of the form $\neg^n \psi$ for some $n \geq 0$.
Let $\varphi^\star$ be $\psi$ if $n$ is even, and $\neg \psi$ otherwise.
Note that $\varphi^\star$ and $(\neg\neg \varphi)^\star$ are identical. 
Since the mapping $\star$ is primitive recursive, we may expand the language of $\PA$ by a function symbol for $\star$ and use the same symbol $\star$ for it in what follows.

By using the formalized double recursion theorem, we will define the primitive recursive functions $h_0$ and $g_0$.
Before defining these functions, we define formulas $\Prf_{g_0}(x,y)$, $\PRR_{g_0}(x)$, and $\PRD_{g_0}(x)$ by using the function $g_0$ as follows: 

\begin{itemize}
\item 
$\Prf_{g_0}(x,y) : \equiv (g_0(y) = x)$.
\item $\PRR_{g_0}(x) : \equiv \exists y (\Prf_{g_0}(x,y) \wedge \forall z \leq y  \neg \Prf_{g_0}(\dot{\neg}x,z) )$.
\item$\PRD_{g_0}(x) : \equiv \exists y (\Prf_{g_0}(x^\star,y) \wedge \forall z \leq y  \neg \Prf_{g_0}((\dot{\neg} x)^\star ,z) )$. 
\end{itemize}
Notice that the formulas $\PRR_{g_0}(x)$ and $\PRD_{g_0}(x)$ differ in how they handle negated formulas. 
\begin{itemize}
\item $\PRR_{g_0}(\gn{\neg \varphi}) \equiv \exists y (\Prf_{g_0}(\gn{\neg \varphi},y) \wedge \forall z \leq y  \neg \Prf_{g_0}(\gn{\neg \neg \varphi},z) )$.
\item$\PRD_{g_0}(\gn{\neg \varphi}) \equiv \exists y (\Prf_{g_0}(\gn{(\neg \varphi)^\star},y) \wedge \forall z \leq y  \neg \Prf_{g_0}(\gn{\varphi^\star} ,z) )$. 
\end{itemize}
Here, `R' and `A' stand for `Rosser' and `Arai', respectively. 
Originally, Arai \cite{Arai} used an operation which maps every formula to the negation normal form of it instead of our operation $\star$.

Let $\lambda_0(x)$ be the formula $\exists y (h_0(y)=x)$. 
Let $x \in \Img(f_0)$ be a $\Delta_1(\PA)$ formula saying that ``$x$ is of the form $f_0(B)$ for some $B \in \MF$''.
We define the arithmetical interpretation $f_0$ and the formula $\PR^{\dagger}_{g_0}(x)$ as follows:
\begin{itemize}
    \item $f_0(p) : \equiv \exists x \exists y( x \in W_y \wedge \lambda_0(x) \wedge x \neq 0 \wedge x \Vdash_y p)$, 
    \item $\PR^{\dagger}_{g_0}(x) : \equiv  (x \in \Img(f_0)  \land \PRR_{g_0}(x))  \lor (x \notin \Img(f_0) \land \PRD_{g_0}(x))$. 
\end{itemize}

Let us indicate the dependencies involved in the simultaneous definition of $h_0$ and $g_0$.
The function $h_0$ determines the formula $\lambda_0$, and hence the
arithmetical interpretation $f_0$. 
The function $g_0$ determines $\PRR_{g_0}$ and $\PRD_{g_0}$. 
Together $f_0$ and $g_0$ determine $\PR^{\dagger}_{g_0}$. 
In the construction below, the definition of $h_0$ will refer to these objects, and the definition of $g_0$ will in turn refer to $h_0$ and to formulas defined using $\PR^{\dagger}_{g_0}$. 
Thus the definitions of $h_0$ and $g_0$ are mutually dependent, and this is why
we use the formalized double recursion theorem.

We introduce the following notion, which generalizes the iteration of $\PR_{g_0}^{\dagger}(x)$ up to the operation $\star$.  
We say that a sequence $(\sigma_0, \ldots, \sigma_n)$ of $\LA$-formulas is a \emph{$\star$-iteration} if $\sigma_{i+1}^\star \equiv \PR^{\dagger}_{g_0}(\gn{\sigma_i})$ for all $i < n$.
For example, we define $\sigma_1$ and $\sigma_2$ as follows:  
\begin{itemize}
    \item $\sigma_1 : \equiv \neg \neg \neg \neg \PR_{g_0}^\dagger(\gn{\rho})$, 
    \item $\sigma_2 : \equiv \neg \neg \PR_{g_0}^\dagger(\gn{\sigma_1})$.  
\end{itemize}
Then, the sequence $(\rho, \sigma_1, \sigma_2)$ is a $\star$-iteration. 

The definitions of the functions $h_0$ and $g_0$ are based on the ideas of the construction of a Rosser provability predicate satisfying $\D{3}$ provided by Arai \cite{Arai} and the construction of a function $h'$ provided in the first author's paper \cite{Kogure1}. 
The function $h_0$ is defined in stages by referring to a condition $\Phi(s)$ and a set $J_s$. 
The condition $\Phi(s)$ states that there exist a formula $\psi$, a number $r \geq 1$, and a $\star$-iteration $(\sigma_0, \ldots, \sigma_{r-1})$ satisfying the following six conditions: 
\begin{enumerate}
    \item $\psi, \sigma_{r-1}^\star \in F_s$, 
    \item $\neg \PR^{\dagger}_{g_0}(\gn{\psi}) \in P_{T,s}$, 
    \item $\psi \notin \Img(f_0)$, 
    \item $\psi^\star \equiv (\neg \sigma_{r-1})^\star$, 
    \item $\sigma_0^\star$ is not of the form $\PR^{\dagger}_{g_0}(\gn{\chi})$, 
    \item $\sigma_j^\star \notin P_{T,s}$ for all $j \leq r - 1$. 
\end{enumerate}
In this setting, it follows from the definition of $f_0$ that $(\neg \sigma_j)^\star \notin \Img(f_0)$ for all $j \leq r-1$. 
Also, we may assume that if a formula $\delta$ contains $\num{n}$ as a sub-expression, then the G\"odel number of $\delta$ is larger than $n$. 
So, for each $i < r-1$, the G\"odel number of $\sigma_{i+1}$ is larger than that of $\sigma_i$ because $\sigma_{i+1}^\star \equiv \PR_{g_0}^{\dagger}(\gn{\sigma_i})$. 
In particular, the G\"odel number of $\psi$ is larger than that of $\sigma_i$ for all $i < r-1$. 
So, the first clause of the definition of the condition $\Phi$ guarantees the primitive recursiveness of the condition.

The set $J_s \subseteq \omega \setminus \{0\}$ is defined as follows: 

\begin{align*}
J_s : = \Big\{ j  \in \omega \setminus \{0\} \ \Big| \ 
 & P_{T,s} \tc \neg \lambda_0(\num{j}) \ \text{or}\\
& \exists k \in \omega \setminus \{0\}\  \exists B \in \Sub(A_k) \bigl[ j \in W_k\ \&\ \\
& \bigl(P_{T,s} \tc \forall x\, \alpha_B(x) \wedge (\alpha_B(\num{j}) \to \neg \lambda_0(\num{j})) \ \text{or}\ \\
& P_{T,s} \tc \forall x\, \beta_B(x) \wedge (\beta_B(\num{j}) \to \neg \lambda_0(\num{j})) \bigr) \bigr] \Big\},
\end{align*}
where $\alpha_B(x)$ and $\beta_B(x)$ are defined as follows:
\begin{itemize}
    \item
    $\alpha_B(x) : \equiv \left( x \neq 0 \wedge \exists y \left( x \in W_y \wedge x \Vdash_y B \right) \right) \to \neg \lambda_0(x)$, and
    \item $\beta_B(x) : \equiv \left( x \neq 0 \wedge \exists y \left( x \in W_y \wedge  x \nVdash_y B \right) \right) \to \neg \lambda_0(x)$.
\end{itemize}

We define the function $h_0$ step by step.
The value of $h_0$ starts at $0$, and $h_0(s)$ becomes non-zero only in either of the following cases: when the condition $\Phi(s)$ holds, or when $\Phi(s)$ fails but $J_s$ is nonempty. 
Once it becomes non-zero, it remains unchanged thereafter. 

\begin{itemize}
    \item $h_0(0) = 0$, 
    \item $h_0(s+1) = 
    \begin{cases}
        1 & \text{if } h_0(s) = 0 \ \&\ \Phi(s)\ \text{holds}, \\
        \min J_s & \text{if } h_0(s) = 0 \ \&\ \Phi(s)\ \text{fails to hold}\ \&\ J_s \neq \emptyset, \\
        h_0(s) & \text{otherwise.}
    \end{cases}$
\end{itemize}

Next, we define the function $g_0$ in stages. 
The construction of $g_0$ consists of Procedures 1 and 2. 
The construction starts with Procedure 1, in which $g_0$ outputs theorems of $T$ by referring to $T$-proofs based on the standard proof predicate $\Proof_T(x, y)$ of $T$. 
If, at some stage $s$, the value of $h_0$ changes from $0$ to a
non-zero value, that is, if $h_0(s)=0$ and $h_0(s+1)\neq 0$, then the construction switches to Procedure 2. 
In Procedure 2, the subsequent values of $g_0$ are defined according to whether this change of $h_0$ is caused by
the condition $\Phi(s)$ or by the non-emptiness of $J_s$.

\paragraph{Procedure 1.}

\textsc{Stage} $s$:
\begin{itemize}
\item If $h_0(s+1) =0$,
\begin{equation*}
  g_0(s)  = \begin{cases}
       \varphi & \text{if}\ s\ \text{is a}\ T \text{-proof of}\ \varphi, \\
               0 & \text{otherwise}.
             \end{cases}
  \end{equation*}

Then, go to \textsc{Stage} $s+1$.

\item If $h_0(s+1) \neq 0$, then go to Procedure 2.
\end{itemize}

\paragraph{Procedure 2.}

Let $s$ and $i \neq 0$ be such that $h_0(s) = 0$ and $h_0(s+1)= i$. 
We find $k \in \omega \setminus \{0\}$ such that $i \in W_k$. 
We define the number $u$ and the values of $g_0(s), \ldots, g_0(s + u-1)$ based on how $h_0(s+1)$ becomes non-zero by distinguishing the following two cases: 

\medskip

Case A. $h_0(s+1)$ becomes non-zero because $\Phi(s)$ holds: 
In this case, we find a formula $\psi$, a number $r \geq 1$ and a $\star$-iteration $(\sigma_0, \ldots, \sigma_{r-1})$ witnessing the condition $\Phi(s)$. 
Let $u : = r$, and for each $j \leq r-1$, let 
\[
  g_0(s+j) : = (\neg \sigma_{r-1-j})^\star.
\]
Notice that $g_0(s+j) \notin \Img(f_0)$ for every $j \leq r-1$. 

\medskip

Case B. $h_0(s+1)$ becomes non-zero because $J_s \neq \emptyset$: 
Let $u : = 0$, and we do nothing in this part of the construction. 

\medskip

Next, we define the values of $g_0(s + u + i)$ for $i \geq 0$. 
Let $X$ be the set of all $\LA$-formulas $(\neg \varphi)^\star$ satisfying the following three conditions: 
\begin{enumerate}
\item $\neg \PR^{\dagger}_{g_0}(\gn{\varphi}) \in P_{T,s-1}$,
\item $\varphi \notin \Img(f_0)$,
\item $\varphi^\star \equiv \neg \PR^{\dagger}_{g_0}(\gn{\chi})$ for some $\chi$.
\end{enumerate}
Let $\chi_0, \ldots, \chi_{l-1}$ be an effective enumeration of all elements of $X$. 
Notice that for each $i \leq l -1$, $\chi_i \notin \Img(f_0)$. 
For each $i \leq l-1$, let
\[
g_0(s+u+i) : = \chi_i.
\]
Recall that $\langle \xi_t \rangle$ is the repetition-free effective enumeration of all $\LA$-formulas in ascending order of G\"odel numbers. 
Then, for each $t$, define
\[
g_0 (s+u+l+t) : = \begin{cases}
    0 & \text{if } \xi_t \equiv f_{0}(B) \text{ and } i \nVdash_k \Box B \text{ for some } B \in \MF, \\
    \xi_t & \text{otherwise.}
\end{cases}
\]

This completes the construction of the function $g_0$.

It is not difficult to show that $h_0(i) \leq i$ for all $i \in \omega$, which guarantees that the function $h_0$ is primitive recursive (cf.~\cite[Claim 5.1]{Kogure1}).
The following proposition states some basic properties of $h_0$. 

\begin{cl}\label{Prop:h_0}
\leavevmode
\begin{enumerate}
\item 
$\PA \vdash \forall x \forall y \bigl( 0 < x < y \wedge \lambda_0(x) \to \neg \lambda_0(y) \bigr)$.
\item 
$\PA \vdash \neg \Con_{T} \leftrightarrow \exists x \bigl(\lambda_0(x) \wedge x \neq 0 \bigr)$.
\item 
For each $i \in \omega \setminus \{ 0\}$, $T \nvdash \neg \lambda_0(\num{i})$.
\item 
For each $l \in \omega$, $\PA \vdash \forall x \forall y \bigl( h_0(x) =0 \wedge h_0(x+1)=y \wedge y \neq 0 \to x \geq \num{l} \bigr)$.
\end{enumerate}
\end{cl}
\begin{proof}
1. Immediate from the definition of $h_0$.

\medskip

2. We argue in $\PA$. $(\leftarrow):$ Suppose $\exists x (\lambda_0(x) \wedge x \neq 0)$. 
Let $s$, $i \neq 0$, and $k$ be such that $h_0(s)=0$ and $h_0(s+1)= i \in W_k$.
We distinguish the following three cases: 

\medskip

Case 1. $\Phi(s)$ holds: 
We find a formula $\psi$, a number $r \geq 1$ and a $\star$-iteration $(\sigma_0, \ldots, \sigma_{r-1})$ witnessing the condition $\Phi(s)$. 
Then, $\sigma_{r-1}^\star \notin P_{T,s}$ and $\psi^\star \equiv (\neg \sigma_{r-1})^\star$. 
We have that $(\neg \psi)^\star \equiv \sigma_{r-1}^\star$ is not in $P_{T, s-1} = \{ g_0(0), \ldots, g_0(s-1) \}$. 
By the definition of Case A of Procedure 2 of the construction of $g_0$, we get $g_0(s) = (\neg \sigma_{r-1})^\star \equiv \psi^\star$. 
Thus, 
\[
    \exists y (\Prf_{g_0}(\gn{\psi^\star},y) \wedge \forall z \leq y  \neg \Prf_{g_0}(\gn{(\neg \psi)^\star} ,z) ), 
\]
that is, $\PRD_{g_0}(\gn{\psi})$ holds. 
Since $\psi \notin \Img(f_0)$, we have that $\PR^{\dagger}_{g_0}(\gn{\psi})$ holds.
Since $\PR^{\dagger}_{g_0}(\gn{\psi})$ is $\Sigma_1$, it is proved in $T$. 
Then, $T$ is inconsistent because $\neg \PR^{\dagger}_{g_0}(\gn{\psi}) \in P_{T,s}$. 

\medskip 

Case 2. $J_s \neq \emptyset$ because $P_{T, s} \tc \neg \lambda_0(\num{i})$: 
$\neg \lambda_0(\num{i})$ is provable in $T$.
Since $\lambda_0(\num{i})$ is a true $\Sigma_1$ sentence, $\lambda_0(\num{i})$ is provable in $T$ by formalized $\Sigma_1$-completeness (cf.~H\'ajek and Pudl\'ak \cite{HP}). 
It follows that $T$ is inconsistent.

\medskip

Case 3. $J_s \neq \emptyset$ because $P_{T, s} \tc \forall x \alpha_B(x) \land (\alpha_B(\num{i}) \to \neg \lambda_0(\num{i}))$ for some $B \in \Sub(A_k)$: 
In this case, $\forall x \alpha_B(x)$ and $\alpha_B(\num{i}) \to \neg \lambda_0(\num{i})$ are provable in $T$. Hence, $\alpha_B(\num{i})$ is provable in $T$,
and so is $\neg \lambda_0(\num{i})$. 
It follows that $T$ is inconsistent. 
The inconsistency of $T$ also follows from $P_{T, s} \tc \forall x \beta_B(x) \land (\beta_B(\num{i}) \to \neg \lambda_0(\num{i}))$ in the same way.  

\medskip

$(\to):$  If $T$ is inconsistent, then $\neg \lambda_0(\num{i})$ is $T$-provable for some $i \neq 0$. Let $s$ be a proof of $\neg \lambda_0(\num{i})$.
Then, $P_{T, s} \tc \neg \lambda_0(\num{i})$ holds and we obtain $J_s \neq \emptyset$. 
By a simple argument, it follows that $h_0(s+1) \neq 0$.

\medskip

3. Suppose that there exists an $i \in \omega \setminus \{0\}$ such that $T \vdash \neg \lambda_0(\num{i})$. Let $p$ be a proof of $\neg \lambda_0(\num{i})$ in $T$.
Then, $\neg \lambda_0(\num{i})$ is a t.c.~of $P_{T,p}$ and we obtain $\mathbb{N} \models \exists x \bigl( \lambda_0(x) \wedge x \neq 0 \bigr)$. By Clause 2, we obtain $\mathbb{N} \models \neg \Con_T$. 
This contradicts the consistency of $T$.

\medskip

4. Since $\mathbb{N} \models \Con_T$, for any $l \in \omega$, we obtain $\mathbb{N} \models h_0(\num{l}) =0$. Thus, $\PA \vdash h_0(\num{l})=0$,
and Clause 4 is easily obtained.
\end{proof}

\begin{cl}\label{ND4-1}
$\PA + \Con_T \vdash \forall x (\Prov_T(x) \leftrightarrow \PR_{g_0}^{\dagger}(x))$.
\end{cl}
\begin{proof}
By the construction of Procedure 1 and Claim \ref{Prop:h_0}.2, it follows that $\PA + \Con_T \vdash \Proof_T(x,y) \leftrightarrow \Prf_{g_0}(x,y)$.
Since $\PA + \Con_T$ proves
\begin{itemize}
    \item $\Proof_T(x,y) \to \forall z \leq y \neg \Proof_{T}(\dot{\neg}x,z)$ and 
    \item $\Proof_T(x^\star,y) \to \forall z \leq y \neg \Proof_{T}((\dot{\neg}x)^\star,z)$, 
\end{itemize}
$\PA + \Con_T$ proves $\Prov_T(x) \leftrightarrow \PRR_{g_0}(x)$ and $\Prov_T(x^\star) \leftrightarrow \PRD_{g_0}(x)$. 
Since $\PA \vdash \Prov_T(x) \leftrightarrow \Prov_T(x^\star)$, we get that $\Prov_T(x)$, $\PRR_{g_0}(x)$, and $\PRD_{g_0}(x)$ are equivalent over $\PA + \Con_T$. 
Then
\begin{itemize}
    \item $\PA + \Con_T \vdash x \in \Img(f_0) \to (\Prov_T(x) \leftrightarrow \PR^\dagger_{g_0}(x))$ and 
    \item $\PA + \Con_T \vdash x \notin \Img(f_0) \to (\Prov_T(x) \leftrightarrow \PR^\dagger_{g_0}(x))$. 
\end{itemize}
By the law of excluded middle, we conclude $\PA + \Con_T \vdash \Prov_T(x) \leftrightarrow \PR^\dagger_{g_0}(x)$. 
\end{proof}

From Claim \ref{ND4-1}, we have that our $\PR_{g_0}^\dagger(x)$ is a $\Sigma_1$ provability predicate of $T$. 

\begin{cl}\label{ND4-2}
Let $D \in \MF$.
\begin{enumerate}
\item 
$\PA \vdash \exists x \bigl( x \neq 0 \wedge \lambda_0(x) \wedge \exists y(x \in W_y \wedge  x \Vdash_y D)  \bigr) \to f_0(D)$, that is, $\PA \vdash \exists x \neg \alpha_D(x) \to f_0(D)$. 

\item 
$\PA \vdash \exists x \bigl( x \neq 0 \wedge \lambda_0(x) \wedge \exists y(x \in W_y \wedge  x \nVdash_y D)  \bigr) \to \neg f_0(D)$, that is, $\PA \vdash \exists x \neg \beta_D(x) \to \neg f_0(D)$.
\end{enumerate}
\end{cl}
\begin{proof}
We prove Clauses 1 and 2 simultaneously by induction on the construction of $D$. 
We prove only the case $D \equiv \Box C$.

\medskip

1. If $L \vdash \neg \Box C$, then $\PA \vdash \forall x \forall y ( x \in W_y \to x \nVdash_y \Box C)$ by the formalized soundness of $L$.
Thus, Clause 1 trivially holds.
So we may assume $L \nvdash \neg \Box C$. 
Then, there exists $j \in \omega$ such that $\neg \Box C \equiv A_j$. 
Let $(W_j, \{\prec_{j,B} \}_{B \in \MF}, \Vdash_j)$ be a countermodel of $A_j$ and let $l \in W_j$ be such that $l \nVdash_j \neg \Box C$. 
Since $(W_j, \{\prec_{j,B} \}_{B \in \MF})$ is an $\ND$-frame, we find $w \in W_j$ such that $l \prec_{j, C} w$ and $l \prec_{j,\neg C} w$. 
Since $l \Vdash_j \Box C$, we have $w \Vdash_j C$.
By the induction hypothesis, $\PA \vdash \exists x \neg \alpha_C(x) \to f_0(C)$.
We also have $\PA \vdash \alpha_C(\num{w}) \to \neg \lambda_0(\num{w})$ because $\PA \vdash \num{w} \neq 0 \land \exists y(\num{w} \in W_y \wedge \num{w} \Vdash_y C)$.
Let $p$ and $q$ be $T$-proofs of $\exists x \neg \alpha_C(x) \to f_0(C)$ and $\alpha_C(\num{w}) \to \neg \lambda_0(\num{w})$, respectively.

We work in $\PA$: Suppose that $\lambda_0(i)$, $i \in W_k$, and $i \Vdash_k \Box C$ hold. 
We prove that $\PR_{g_0}^{\dagger}(\gn{f_0(C)})$ holds. 
Let $s$ be such that $h_0(s)=0$ and $h_0(s+1)= i \neq 0$. 
Let $t$ be such that $\xi_t \equiv f_0(C)$. 
Let $u$ and $l$ be as in the definition of Procedure 2 of the construction of $g_0$. 
We obtain $g_0(s + u+ l + t)= \xi_t$ because $i \Vdash_k \Box C$.

Then, it suffices to show that $\neg f_0(C) \notin \{g_0(0), \ldots,  g_0(s+u+l+t) \}$.
Since $\neg f_0(C) \in \Img(f_0)$, we have $\neg f_0(C) \notin \{g_0(s), \ldots, g_0(s+u-1)\}$ because $g_0(s + j) \notin \Img(f_0)$ for every $j \leq u-1$.  
By the same reason, we have $\neg f_0(C) \notin X$, and hence $\neg f_0(C) \notin \{ g_0(s+u), \ldots, g_0(s+u+l-1)\}$.
Since the G\"{o}del number of $\neg f_0(C)$ is larger than that of $f_0(C)$, we have $\neg f_0(C) \notin \{ g_0(s+u+l), \ldots, g_0(s+u+l+t)\}$.

Thus, it suffices to show that $\neg f_0(C) \notin P_{T,s-1} = \{g_0(0), \ldots, g_0(s-1)\}$.
Suppose, toward a contradiction, that $\neg f_0(C) \in P_{T,s-1}$. 
By Claim \ref{Prop:h_0}.4, we have $s > p, q$, and so $\exists x \neg \alpha_C(x) \to f_0(C)$ and $\alpha_C(\num{w}) \to \neg \lambda_0(\num{w})$ are in $P_{T,s-1}$. 
It follows from $P_{T,s-1} \tc \forall x \alpha_C(x)$ that $w \in J_{s-1}$. 
Then, $h_0(s) \neq 0$, a contradiction. 
Therefore, we obtain $\neg f_0(C) \notin \{g_0(0), \ldots,  g_0(s+u+l+t) \}$. 
We conclude that $\PR^{\dagger}_{g_0}(\gn{f_0(C)})$ holds.

\medskip

2. If $L \vdash \Box C$, then $\PA \vdash \forall x \forall y (x \in W_y \to x \Vdash_y \Box C)$, and hence Clause 2 trivially holds.
So, we may assume $L \nvdash \Box C$. 
Then, there exist $j \in \omega \setminus \{0\}$ and a countermodel $(W_j, \{\prec_{j,B} \}_{B \in \MF}, \Vdash_j)$ of $\Box C$.
Then, there exists $r \in W_j$ such that $r \nVdash_j C$.
By the induction hypothesis, $\PA \vdash \exists x \neg \beta_C(x) \to \neg f_0(C)$.
Since $\PA \vdash \num{r} \neq 0 \land \exists y(\num{r} \in W_y \wedge \num{r} \nVdash_y C)$, we obtain $\PA \vdash \beta_C(\num{r}) \to \neg \lambda_0(\num{r})$.
Let $p$ and $q$ be $T$-proofs of $\exists x \neg \beta_C(x) \to \neg f_0(C)$ and $\beta_C(\num{r}) \to \neg \lambda_0(\num{r})$, respectively.

We work in $\PA$: Suppose that $\lambda_0(i)$, $i \in W_k$, and $i \nVdash_k \Box C$ hold. 
Let $s$ be such that $h_0(s)=0$ and $h_0(s+1)= i \neq 0$.
Let $u$ and $l$ be as in the construction of $g_0$. 
We show that $f_0(C)$ is not output by $g_0$.

Suppose, toward a contradiction, that $f_0(C) \in P_{T,s-1}$. 
By Claim \ref{Prop:h_0}.4, we have $s > p, q$, and so $\exists x \neg \beta_C(x) \to \neg f_0(C)$ and $\beta_C(\num{r}) \to \neg \lambda_0(\num{r})$ are in $P_{T,s-1}$. 
It follows from $P_{T,s-1} \tc \forall x \beta_C(x)$ that $r \in J_{s-1}$. 
This contradicts $h_0(s) = 0$. 
Therefore, we obtain $f_0(C) \notin \{g_0(0), \ldots,  g_0(s-1) \}$. 
Since $f_0(C) \in \Img(f_0)$, we have $f_0(C) \notin \{g_0(s), \ldots, g_0(s+u+l-1)\}$. 
Since $i \nVdash_k \Box C$, we obtain $g_0(s+u+l+t) \neq f_0(C)$ for all $t \geq 0$. 
Therefore, we have shown that $\forall y \neg \Prf_{g_0}(\gn{f_0(C)}, y)$ holds.
In particular, $\neg \PR^{\dagger}_{g_0}(\gn{f_0(C)})$ holds.
\end{proof}

We prove that the schematic consistency statement based on $\PR_{g_0}^{\dagger}(x)$ is provable in $\PA$. 

\begin{cl}\label{ND4-3}
For any $\LA$-formula $\varphi$,
$\PA \vdash \neg (\PR_{g_0}^{\dagger}(\gn{\varphi}) \wedge \PR_{g_0}^{\dagger}(\gn{\neg \varphi}))$.
\end{cl}
\begin{proof}
We distinguish the following two cases:

\medskip

Case 1. $\varphi \in \Img(f_0)$: 
Suppose that $\varphi \equiv f_0(B)$ for some $B \in \MF$. 
Since $\PA + \Con_T \vdash \Prov_T(\gn{\varphi}) \to \neg \Prov_T(\gn{\neg \varphi})$, 
by Claim \ref{ND4-1}, we obtain $\PA + \Con_T \vdash \PR^{\dagger}_{g_0}(\gn{\varphi}) \to \neg \PR^{\dagger}_{g_0}(\gn{\neg \varphi})$. 
So, by Claim \ref{Prop:h_0}.2, it suffices to prove $\PA + \exists x (x \neq 0 \land \lambda_0(x)) \vdash f_0(\Box B) \to \neg f_0(\Box \neg B)$. 

By Claim \ref{ND4-2}.2, 
\[
    \PA \vdash x \neq 0 \land \lambda_0(x) \land x \in W_y \land x \nVdash_y \Box B \to \neg f_0(\Box B). 
\]
Then, 
\[
    \PA \vdash x \neq 0 \land \lambda_0(x) \land x \in W_y \land f_0(\Box B) \to x \Vdash_y \Box B. 
\]
Since $\PA$ proves that every $W_k$ validates $\ND$, 
\[
    \PA \vdash x \neq 0 \land \lambda_0(x) \land x \in W_y \land f_0(\Box B) \to x \nVdash_y \Box \neg B. 
\]
By combining this with Claim \ref{ND4-2}.2, 
\[
    \PA \vdash x \neq 0 \land \lambda_0(x) \land x \in W_y \land f_0(\Box B) \to \neg f_0(\Box \neg B). 
\]
We conclude $\PA + \exists x(x \neq 0 \land \lambda_0(x)) \vdash f_0(\Box B) \to \neg f_0(\Box \neg B)$.

\medskip

Case 2. $\varphi \notin \Img(f_0)$:
It follows that
\begin{itemize}
    \item $\PA \vdash \PR^{\dagger}_{g_0}(\gn{\varphi}) \leftrightarrow \exists y (\Prf_{g_0}(\gn{ \varphi^\star},y) \wedge \forall z \leq y \neg \Prf_{g_0}(\gn{(\neg \varphi)^\star},z))$ and 
    \item $\PA \vdash \PR^{\dagger}_{g_0}(\gn{\neg \varphi}) \leftrightarrow \exists y (\Prf_{g_0}(\gn{(\neg \varphi)^\star},y) \wedge \forall z \leq y \neg \Prf_{g_0}(\gn{\varphi^\star},z))$.
\end{itemize}
Therefore, $\PA \vdash \PR^{\dagger}_{g_0}(\gn{\varphi}) \to \neg \PR^{\dagger}_{g_0}(\gn{\neg \varphi})$ follows from an easy witness comparison argument.
\end{proof}

\begin{cl}\label{ND4-4}
Let $\varphi$ be an $\LA$-formula. 
$\PA$ proves the following statement:
``If $h_0(s)=0$, $h_0(s+1) \neq 0$, $\varphi \notin \Img(f_0)$, and $\neg \PR^{\dagger}_{g_0} (\gn{\varphi}) \in P_{T,s-1}$, then the following hold: 
\begin{enumerate}
    \item[\textup{(i)}] 
    If $\varphi^\star$ is not of the form $\neg \PR^{\dagger}_{g_0}(\gn{\chi})$, then $(\neg \varphi)^\star \in P_{T,s-1}$.

    \item[\textup{(ii)}] 
    If there exist $r \geq 1$ and a $\star$-iteration $(\sigma_0, \ldots, \sigma_r)$ such that $\varphi \equiv \sigma_r$ and $\sigma_0^\star$ is not of the form $\PR^{\dagger}_{g_0}(\gn{\chi})$, then $(\neg \sigma_i)^\star \in P_{T,s-1}$ for all $i$ with $1 \leq i \leq r$.

    \item[\textup{(iii)}]
If $\psi$, $r \geq 1$, and $(\sigma_0, \ldots, \sigma_{r-1})$ witness the condition $\Phi(s)$, then $\varphi^\star \not \equiv (\neg \sigma_i)^\star$ for all $i < r$.

\item[\textup{(iv)}] 
$\varphi^\star \notin P_{T,s-1}$."
\end{enumerate}    
\end{cl}
\begin{proof}
We reason in $\PA$. 
Suppose $h_0(s)=0$, $h_0(s+1) \neq 0$, $\varphi \notin \Img(f_0)$, and $\neg \PR^{\dagger}_{g_0}(\gn{\varphi}) \in P_{T,s-1}$. 

\medskip

(i). Suppose, towards a contradiction, that $\varphi^\star$ is not of the form $\neg \PR^{\dagger}_{g_0}(\gn{\chi})$ and $(\neg \varphi)^\star \notin P_{T,s-1}$.
     Then, $\varphi$, $r = 1$, and the $\star$-iteration $(\neg \varphi)$ witness the condition $\Phi(s-1)$. 
This contradicts $h_0(s)=0$.

\medskip 

(ii). Suppose that there exist $r \geq 1$ and a $\star$-iteration $(\sigma_0, \ldots, \sigma_r)$ such that $\varphi \equiv \sigma_r$ and $\sigma_0^\star$ is not of the form $\PR^{\dagger}_{g_0}(\gn{\chi})$.
We prove $(\neg \sigma_{r-j})^\star \in P_{T,s-1}$ for all $j \leq r-1$ by induction on $j$.
We prove the base case $j=0$. 
Since $\varphi^\star \equiv \PR_{g_0}^{\dagger}(\gn{\sigma_{r-1}})$, it is not of the form $\neg \PR^{\dagger}_{g_0}(\gn{\chi})$. 
By (i), we obtain $(\neg \sigma_r)^\star \equiv (\neg \varphi)^\star\in P_{T,s-1}$.
We prove the induction step. 
Suppose $j+1 \leq r-1$ and $(\neg \sigma_{r-j})^\star \in P_{T,s-1}$. 
We get $\sigma_{r-(j+1)} \notin \Img(f_0)$ because $\varphi \notin \Img(f_0)$. 
Also, $\neg \PR^{\dagger}_{g_0}(\gn{\sigma_{r-(j+1)}}) \equiv (\neg \sigma_{r-j})^\star \in P_{T,s-1}$. 
Since $r-(j+1) \geq 1$, we have that $(\sigma_{r-(j+1)})^\star$ is not of the form $\neg \PR^{\dagger}_{g_0}(\gn{\chi})$. 
We may assume that we are in the position where we have already proved (i) for the $\LA$-formula $\sigma_{r-(j+1)}$. 
Therefore, $(\neg \sigma_{r-(j+1)})^\star \in P_{T,s-1}$.

\medskip

(iii). Suppose that $\psi$, $r \geq 1$, and $(\sigma_0, \ldots, \sigma_{r-1})$ witness the condition $\Phi(s)$. 
Assume, towards a contradiction, that there exists $i_0 < r$ such that $\varphi^\star \equiv (\neg \sigma_{i_0})^\star$.
By the condition $\Phi(s)$, for each $j \leq {i_0}$, we have $\sigma_j^\star \notin P_{T,s-1}$. 
Since $\neg \PR^{\dagger}_{g_0}(\gn{\varphi}) \in P_{T,s-1}$ and $\varphi \notin \Img(f_0)$, we have that $\varphi$, $i_0+1$, and $(\sigma_0, \ldots, \sigma_{i_0})$ witness the condition $\Phi(s-1)$. 
This contradicts $h_0(s)=0$.
Thus, $\varphi^\star \not \equiv (\neg \sigma_i)^\star$ for all $i < r$.

\medskip

(iv). We prove $\varphi^\star \notin P_{T,s-1}$ by induction on the G\"{o}del number of $\varphi$.
Suppose that for any $\psi$, if 
\begin{itemize}
    \item the G\"{o}del number of $\psi$ is smaller than that of $\varphi$,
    \item 
    $\psi \notin \Img(f_0)$, and 
    \item 
    $\neg \PR^{\dagger}_{g_0}(\gn{\psi}) \in P_{T,s-1}$,
\end{itemize}
then $\psi^\star \notin P_{T,s-1}$.
Suppose, towards a contradiction, that $\varphi^\star \in P_{T,s-1}$.
If $(\neg \varphi)^\star \in P_{T,s-1}$, then $P_{T,s-1}$ is inconsistent, and thus $P_{T,s-1} \tc \neg \lambda_0(\num{i})$ for all $i$. 
This contradicts $h_0(s)=0$. 
Hence, we get $(\neg \varphi)^\star \notin P_{T, s-1}$. 
By (i), $\varphi^\star$ is of the form $\neg \PR^{\dagger}_{g_0}(\gn{\eta})$ for some $\eta$. 
Let $r \geq 1$ be the largest number such that there exists a $\star$-iteration $(\sigma_0, \ldots, \sigma_r)$ such that $\varphi^\star \equiv (\neg \sigma_r)^\star$. 
Then, for such $\sigma_0$, we have that $\sigma_0^\star$ is not of the form $\PR^{\dagger}_{g_0}(\gn{\chi})$. 
In the case $r > 1$, since $\sigma_{r-1} \notin \Img(f_0)$ and $\neg \PR^{\dagger}_{g_0}(\gn{\sigma_{r-1}}) \equiv (\neg \sigma_r)^\star \equiv \varphi^\star \in P_{T, s-1}$, by applying (ii) to the $\LA$-formula $\sigma_{r-1}$ with the $\star$-iteration $(\sigma_0, \ldots, \sigma_{r-1})$, we have that $(\neg \sigma_{i})^\star \in P_{T,s-1}$ for all $i$ with $1 \leq i \leq r-1$. 
So, regardless of whether $r > 1$ or $r = 1$, we obtain that $(\neg \sigma_{i})^\star \in P_{T,s-1}$ for all $i$ with $1 \leq i \leq r$. 

Since $h_0(s) = 0$, we have that $P_{T, s-1}$ is not inconsistent, and hence $\sigma_j^\star \notin P_{T, s-1}$ for all $j$ with $1 \leq j \leq r$. 
We have $\sigma_0 \notin \Img(f_0)$ and $\neg \PR_{g_0}^{\dagger}(\gn{\sigma_0}) \equiv (\neg \sigma_1)^\star \in P_{T, s-1}$. 
Since the G\"{o}del number of $\sigma_0$ is smaller than that of $\varphi$, by the induction hypothesis, $\sigma_0^\star \notin P_{T,s-1}$.
So, we have shown that $\sigma_j^\star \notin P_{T, s-1}$ for all $j \leq r$. 
Therefore, it is shown that $\varphi$, $r+1$, and $(\sigma_0, \ldots, \sigma_r)$ witness the condition $\Phi(s-1)$. 
This contradicts $h_0(s)=0$.
We conclude $\varphi^\star \notin P_{T,s-1}$.
\end{proof}

\begin{cl}\label{ND4-5}
Suppose $L = \NDF$. 
Then, for any $\LA$-formula $\varphi$, $\PA \vdash \PR_{g_0}^{\dagger}(\gn{\varphi}) \to \PR^{\dagger}_{g_0}(\gn{\PR_{g_0}^{\dagger}(\gn{\varphi})})$.
\end{cl}
\begin{proof}
Since $\PR^{\dagger}_{g_0}(x)$ is $\Sigma_1$, we obtain $\PA \vdash \PR^{\dagger}_{g_0}(\gn{\varphi}) \to \Prov_T(\gn{\PR^{\dagger}_{g_0}(\gn{\varphi})})$.
By Claim \ref{ND4-1}, $\PA + \Con_T \vdash \PR^{\dagger}_{g_0}(\gn{\varphi}) \to \PR^{\dagger}_{g_0}(\gn{\PR^{\dagger}_{g_0}(\gn{\varphi})})$.
So, by Claim \ref{Prop:h_0}.2, it suffices to prove $\PA + \exists x (x \neq 0 \land \lambda_0(x)) \vdash \PR^{\dagger}_{g_0}(\gn{\varphi}) \to \PR^{\dagger}_{g_0}(\gn{\PR^{\dagger}_{g_0}(\gn{\varphi})})$.
We distinguish the following two cases: 

\medskip

Case 1. $\varphi \in \Img(f_0)$: 
Let $\varphi \equiv f_0(B)$ for some $B \in \MF$.
By Claim \ref{ND4-2}.2,  
\[
    \PA \vdash x \neq 0 \land \lambda_0(x) \land x \in W_y \land x \nVdash_y \Box B \to \neg f_0(\Box B), 
\]
and so
\[
    \PA \vdash x \neq 0 \land \lambda_0(x) \land x \in W_y \land f_0(\Box B) \to x \Vdash_y \Box B.  
\]
Since $\PA$ proves that every $W_k$ validates $\NDF$, 
\[
    \PA \vdash x \neq 0 \land \lambda_0(x) \land x \in W_y \land f_0(\Box B) \to x \Vdash_y \Box \Box B.  
\]
By Claim \ref{ND4-2}.2,  
\[
    \PA \vdash x \neq 0 \land \lambda_0(x) \land x \in W_y \land x \Vdash_y \Box \Box B \to f_0(\Box \Box B). 
\]
By combining them, 
\[
    \PA \vdash x \neq 0 \land \lambda_0(x) \land x \in W_y \land f_0(\Box B) \to f_0(\Box \Box B).  
\]
We conclude $\PA + \exists x (x \neq 0 \land \lambda_0(x)) \vdash f_0(\Box B) \to f_0(\Box \Box B)$. 

\medskip

Case 2. $\varphi \notin \Img(f_0)$:         
We work in $\PA + \exists x (\lambda_0(x) \wedge x \neq 0)$:
Let $s$, $i \neq 0$, and $k$ be such that $h_0(s)=0$ and $h_0(s+1)=i \in W_k$.
Suppose that $\neg \PR^{\dagger}_{g_0}(\gn{\PR^{\dagger}_{g_0}(\gn{\varphi})})$ holds.
Since $\PR^{\dagger}_{g_0}(\gn{\varphi}) \notin \Img(f_0)$, this means that $\neg \PRD_{g_0}(\gn{\PR^{\dagger}_{g_0}(\gn{\varphi})})$ holds.
We prove that $\neg \PRD_{g_0}(\gn{\varphi})$ holds.
Let $u$ and $l$ be as in Procedure 2 of the construction of $g_0$. 
Let $t_0$ and $t_1$ be such that $\xi_{t_0} \equiv \PR^{\dagger}_{g_0}(\gn{\varphi})$ and $\xi_{t_1} \equiv \neg \PR^{\dagger}_{g_0}(\gn{\varphi})$. 
Then, $t_0 < t_1$. 
Since $\PR^{\dagger}_{g_0}(\gn{\varphi}), \neg \PR^{\dagger}_{g_0}(\gn{\varphi}) \notin \Img(f_0)$, we have that $g_0(s+u+l+t_0) = \PR^{\dagger}_{g_0}(\gn{\varphi})$ and $g_0(s+u+l+t_1) = \neg \PR^{\dagger}_{g_0}(\gn{\varphi})$. 
Since $\neg \PR^{\dagger}_{g_0}(\gn{\PR^{\dagger}_{g_0}(\gn{\varphi})})$ holds, it follows that this is not the first output of $\neg \PR^{\dagger}_{g_0}(\gn{\varphi})$. 
Since $\neg \PR^{\dagger}_{g_0}(\gn{\varphi})$ is not of the form $\PR^{\dagger}_{g_0}(\gn{\chi})$, it follows that $\neg \PR^{\dagger}_{g_0}(\gn{\varphi}) \notin X$.
Therefore, we have $\neg \PR^{\dagger}_{g_0}(\gn{\varphi}) \in \{g_0(0), \ldots, g_0(s + u-1)\}$. 
We distinguish the following two cases: 

\medskip 

Case 2.1. $\neg \PR^{\dagger}_{g_0}(\gn{\varphi}) \in \{g_0(0), \ldots, g_0(s-1)\}$: In this case, $\neg \PR^{\dagger}_{g_0}(\gn{\varphi}) \in P_{T,s-1}$. 
By (iv) of Claim \ref{ND4-4}, we have $\varphi^\star \notin P_{T,s-1}$. 
If $\varphi^\star$ is not of the form $\neg \PR^{\dagger}_{g_0}(\gn{\chi})$, then by (i) of Claim \ref{ND4-4}, we have $(\neg \varphi)^\star \in P_{T,s-1}$.
It follows that $\neg \PRD_{g_0}(\gn{\varphi})$ holds.
If $\varphi^\star \equiv \neg \PR^{\dagger}_{g_0}(\gn{\chi})$ for some $\chi$, then  $(\neg \varphi)^\star \in X$ and $\varphi^\star \notin X$.
By (iii) of Claim \ref{ND4-4}, even if the condition $\Phi(s)$ holds, we have $\varphi^\star \notin \{g_0(0), \ldots, g_0(s+u+l-1)\}$. 
Therefore, we get that $\neg \PRD_{g_0}(\gn{\varphi})$ holds.

\medskip 

Case 2.2. $\neg \PR^{\dagger}_{g_0}(\gn{\varphi}) \in \{g_0(s), \ldots, g_0(s+u-1)\}$: 
This happens when the condition $\Phi(s)$ holds. 
Let $\psi$, $r \geq 1$, and $(\sigma_0, \ldots, \sigma_{r-1})$ witness the condition $\Phi(s)$. 
In this case, we find $j \leq r-1$ such that $g_0(s+j) = \neg \PR^{\dagger}_{g_0}(\gn{\varphi}) \equiv (\neg \sigma_{r-1-j})^\star$. 
Since $\sigma_0^\star \not \equiv  \PR^{\dagger}_{g_0}(\gn{\varphi})$, we have $r - 1 - j \geq 1$. 
Then, $\PR_{g_0}^{\dagger}(\gn{\varphi}) \equiv \sigma_{r - 1 -j}^\star \equiv \PR_{g_0}^{\dagger}(\gn{\sigma_{r- 1 - (j+1)}})$, and hence $\varphi \equiv \sigma_{r - 1 - (j+1)}$. 
We get
\[
    g_0(s+j+1)= (\neg \sigma_{r-1-(j+1)})^\star \equiv (\neg \varphi)^\star.
\]
Since the G\"{o}del number of $(\neg \sigma_{r-1-l})^\star$ for $l \leq j$ is larger than that of $\varphi^\star$, we have $\varphi^\star \notin \{ g_0(s), \ldots, g_0(s+j) \}$.
Also by the condition $\Phi(s)$, we obtain $\varphi^\star \equiv (\sigma_{r-1-(j+1)})^\star \notin P_{T,s-1} = \{g_0(0), \ldots, g_0(s-1)\}$.
We conclude that $\neg \PRD_{g_0}(\gn{\varphi})$ holds.
\end{proof}

We finish our proof of Theorem \ref{thmND4}. 
Clause 1 and the implication $(\Rightarrow)$ of Clause 2 trivially hold by Claim \ref{ND4-3} and Claim \ref{ND4-5}.  

We prove the implication $(\Leftarrow)$ of Clause 2. Suppose $L \nvdash A$. 
Then, there exists $k \in \omega$ such that $A \equiv A_k$.
Let $\bigl( W_k, \{\prec_{k,B}\}_{B \in \MF}, \Vdash_k  \bigr)$ be a countermodel of $A_k$, and let $i \in W_k$ be $i \nVdash_k A_k$.
Hence, we obtain 
\[
\PA \vdash \num{i} \neq 0 \wedge \exists y \bigl( \num{i} \in W_y \wedge \num{i} \nVdash_y A \bigr)
\]
and it follows that $\PA \vdash \beta_A (\num{i}) \to \neg \lambda_0(\num{i})$.
By the contrapositive of Claim \ref{ND4-2}.2, we get
$\PA \vdash f_0(A)  \to \forall x \beta_A (x)$, which implies $\PA \vdash f_0(A)  \to  \beta_A (\num{i})$.
Then, we obtain $\PA \vdash f_0(A) \to \neg \lambda_0(\num{i})$.
It follows from Claim \ref{Prop:h_0}.3 that $T \nvdash f_0(A)$.
\end{proof}

\begin{cor}[The arithmetical completeness of $\ND$]\label{AC_ND}
\begin{align*}
    \ND & = \bigcap \{ \PL(\PR_T) \mid \PR_T(x) \text{ is a provability predicate satisfying }  T \vdash \ConS_T \},\\
    & = \bigcap \{ \PL(\PR_T) \mid \PR_T(x) \text{ is a } \Sigma_1 \text{ provability predicate satisfying }  T \vdash \ConS_T \}.
\end{align*}
Moreover, there exists a $\Sigma_1$ provability predicate $\PR_T(x)$ of $T$ such that $\ND = \PL(\PR_T)$.
\end{cor}

\begin{cor}[The arithmetical completeness of $\NDF$]\label{AC_ND4}
\begin{align*}
    \NDF & = \bigcap \{ \PL(\PR_T) \mid \PR_T(x) \text{ satisfies } \D{3} \text{ and } T \vdash \ConS_T \},\\
    & = \bigcap \{ \PL(\PR_T) \mid \PR_T(x) \text{ is } \Sigma_1 \text{ and satisfies } \D{3} \text{ and } T \vdash \ConS_T \}.
\end{align*}
Moreover, there exists a $\Sigma_1$ provability predicate $\PR_T(x)$ of $T$ such that $\NDF = \PL(\PR_T)$.
\end{cor}

\section{The arithmetical completeness of $\NP$ and $\NPF$}\label{AC2}

In this section, we prove the arithmetical completeness theorems for $\NP$ and $\NPF$.

\begin{thm}\label{thmNP4}
Let $L \in \{ \NP, \NPF \}$. 
There exists a $\Sigma_1$ provability predicate $\PR_T(x)$ of $T$ satisfying the following properties: 
\begin{enumerate}
    \item \textup{(Arithmetical soundness)} 
    For any $A \in \MF$ and any arithmetical interpretation $f$ based on $\PR_T(x)$, if $L \vdash A$, then $\PA \vdash f(A)$. 
    \item \textup{(Uniform arithmetical completeness)} 
    There exists an arithmetical interpretation $f$ based on $\PR_T(x)$ such that for any $A \in \MF$, $L \vdash A$ if and only if $T \vdash f(A)$.
\end{enumerate}

\end{thm}
\begin{proof}
Let $L \in \{ \NP, \NPF \}$.
Let $\langle A_k  \rangle_{k \in \omega}$ be a primitive recursive enumeration of all $L$-unprovable modal formulas.
As in the proof of \ref{thmND4}, for each $k \in \omega$, we can primitive recursively construct a finite $L$-model $\bigl( W_k, \{ \prec_{k,B}\}_{B \in \MF}, \Vdash_k   \bigr)$ falsifying $A_k$.
We may assume that the sets $\{ W_k \}_{k \in \omega}$ are pairwise disjoint subsets of $\omega$ and $\bigcup_{k \in \omega} W_k = \omega \setminus \{0\}$.

By using the formalized recursion theorem, we define the primitive recursive functions $h_1$ and $g_1$.
Here, the definitions of $h_1$ and $g_1$ are simpler than those of $h_0$ and $g_0$, respectively.
The definition of $h_1$ refers only to the following set $J_s'$: 

\begin{align*}
J_s' : = \Big\{ j  \in \omega \setminus \{0\} \ \Big| \ 
 & P_{T,s} \tc \neg \lambda_1(\num{j}) \ \text{or}\\
& \exists k \in \omega \setminus \{0\}\  \exists B \in \Sub(A_k) \bigl[ j \in W_k\ \&\ \\
& \bigl(P_{T,s} \tc \forall x\, \alpha_B'(x) \wedge (\alpha_B'(\num{j}) \to \neg \lambda_1(\num{j})) \ \text{or}\ \\
& P_{T,s} \tc \forall x\, \beta_B'(x) \wedge (\beta_B'(\num{j}) \to \neg \lambda_1(\num{j})) \bigr) \bigr] \Big\}.
\end{align*}

The formulas appearing in the definition are given as follows.
\begin{itemize}
    \item $\lambda_1(x) : \equiv \exists y \left( h_1(y) = x \right)$.
  \item
    $\alpha_B'(x) : \equiv \left( x \neq 0 \wedge \exists y \left( x \in W_y \wedge x \Vdash_y B \right) \right) \to \neg \lambda_1(x)$. 
    \item $\beta_B'(x) : \equiv \left( x \neq 0 \wedge \exists y \left( x \in W_y \wedge  x \nVdash_y B \right) \right) \to \neg \lambda_1(x)$.
\end{itemize}

We define the function $h_1$ as follows: 

\begin{itemize}
    \item $h_1(0) = 0$, 
    \item $h_1(s+1) = 
    \begin{cases}
        \min J'_s & \text{if } h_1(s) = 0 \ \&\  J'_s \neq \emptyset, \\
        h_1(s) & \text{otherwise.}
    \end{cases}$
\end{itemize}

The following proposition is proved in a similar way as Claim \ref{Prop:h_0}. 

\begin{cl}\label{Prop:h_1}
\leavevmode
\begin{enumerate}
\item 
$\PA \vdash \forall x \forall y \bigl( 0 < x < y \wedge \lambda_1(x) \to \neg \lambda_1(y) \bigr)$.
\item 
$\PA \vdash \neg \Con_{T} \leftrightarrow \exists x \bigl(\lambda_1(x) \wedge x \neq 0 \bigr)$.
\item 
For each $i \in \omega \setminus \{ 0\}$, $T \nvdash \neg \lambda_1(\num{i})$.
\item 
For each $l \in \omega$, $\PA \vdash \forall x \forall y \bigl( h_1(x) =0 \wedge h_1(x+1)=y \wedge y \neq 0 \to x \geq \num{l} \bigr)$.
\end{enumerate}
\end{cl}

Next, we define the function $g_1$. 
The definition of $g_1$ is also considerably simpler than that of $g_0$. 
In the definition of $g_1$, we use the formula $\PR_{g_1}(x) : \equiv \exists y (g_1(y)=x)$ and 
the arithmetical interpretation $f_1$ based on $\PR_{g_1}(x)$ defined as follows: 
$f_1(p) : \equiv \exists x \exists y (x \in W_y \wedge \lambda_1(x) \wedge x \neq 0 \wedge x \Vdash_y p)$.

\paragraph{Procedure 1.}

\textsc{Stage} $s$:
\begin{itemize}
\item If $h_1(s+1) =0$,
\begin{equation*}
  g_1(s)  = \begin{cases}
       \varphi & \text{if}\ s\ \text{is a}\ T \text{-proof of}\ \varphi, \\
               0 & \text{otherwise}.
             \end{cases}
  \end{equation*} 

Then, go to \textsc{Stage} $s+1$.

\item If $h_1(s+1) \neq 0$, go to Procedure 2.
\end{itemize}

\paragraph{Procedure 2.}
Let $s$ and $i \neq 0$ be such that $h_1(s) = 0$ and $h_1(s+1)= i$. 
We find $k \in \omega \setminus \{0\}$ such that $i \in W_k$. 
For each $t$, 
\[
g_1 (s+t) : = \begin{cases}
    0 & \text{if } \xi_t \equiv f_1(B) \text{ and } i \nVdash_k \Box B \text{ for some } B \in \MF, \\
    \xi_t & \text{otherwise.}
\end{cases}
\]
The construction of $g_1$ has been finished.
The following claim guarantees that our formula $\PR_{g_1}(x)$ is a $\Sigma_1$ provability predicate of $T$, which is proved in a similar way as Claim \ref{ND4-1}.

\begin{cl}\label{NP4-1}
For any $\LA$-formula $\varphi$,
$\PA + \Con_T \vdash \Prov_T(\gn{\varphi}) \leftrightarrow \PR_{g_1}(\gn{\varphi})$.
\end{cl}

\begin{cl}\label{NP4-2}
Let $B \in \MF$.
\begin{enumerate}
\item 
$\PA \vdash \exists x \bigl( x \neq 0 \wedge \lambda_1(x) \wedge \exists y(x \in W_y \wedge  x \Vdash_y B)  \bigr) \to f_1(B)$.
\item 
$\PA \vdash \exists x \bigl( x \neq 0 \wedge \lambda_1(x) \wedge \exists y(x \in W_y \wedge  x \nVdash_y B)  \bigr) \to \neg f_1(B)$.
\end{enumerate}
\end{cl}
\begin{proof}
We prove the claim by induction on the construction of $B$.
We prove only the case $B \equiv \Box C$.
The second clause is proved similarly as in the proof of Claim \ref{ND4-2}.
We only prove the first clause.

We work in $\PA$: Let $s$, $i \neq 0$, and $k$ be such that $h_1(s)=0$ and $h_1(s+1)= i  \in W_k$. Suppose $i \Vdash_k \Box C$.
Let $\xi_t$ be $f_1(C)$. Then, we obtain $g_1(s+t)= \xi_t$, that is, $\PR_{g_1}(\gn{f_1(C)})$ holds.
\end{proof}

\begin{cl}\label{NP4-3}
    $\PA \vdash \neg \PR_{g_1}(\gn{0=1})$.
\end{cl}
\begin{proof}
Since $\PA + \Con_T \vdash \neg \Prov_T(\gn{0=1})$, by Claim \ref{NP4-1}, it suffices to prove $\PA + \exists x (\lambda_1(x) \wedge x \neq 0) \vdash \neg \PR_{g_1}(\gn{0=1})$.
We reason in $\PA + \exists x (\lambda_1(x) \wedge x \neq 0)$: Let $s$, $i \neq 0$, and $k$ be such that $h_1(s)=0$ and $h_1(s+1) = i \in W_k$.
Since $(W_k, \{ \prec_{k,C}  \}_{C \in \MF})$ is an $\NP$-frame, we obtain $i \nVdash_k \Box \bot$. 
Thus, by Claim \ref{NP4-2}.2, $\neg \PR_{g_1}(\gn{0=1})$ holds.
\end{proof}

\begin{cl}\label{NP4-4}
If $L = \NPF$, then for any $\LA$-formula $\varphi$,
$\PA \vdash \PR_{g_1}(\gn{\varphi}) \to  \PR_{g_1}(\gn{\PR_{g_1}(\gn{\varphi})})$.
\end{cl}
\begin{proof}
As in the proof of Claim \ref{ND4-5}, we obtain 
$\PA + \Con_T \vdash \PR_{g_1}(\gn{\varphi}) \to \PR_{g_1}(\gn{\PR_{g_1}(\gn{\varphi})})$.
So, it suffices to prove $\PA + \exists x (\lambda_1(x) \wedge x \neq 0) \vdash \PR_{g_1}(\gn{\varphi}) \to \PR_{g_1}(\gn{\PR_{g_1}(\gn{\varphi})})$.
We distinguish the following two cases: 

\medskip

Case 1. $\varphi \in \Img(f_1)$: 
The proof of this case is the same as that of Claim \ref{ND4-5}.

\medskip

Case 2. $\varphi \notin \Img(f_1)$: 
We work in $\PA + \exists x (\lambda_1(x) \wedge x \neq 0)$: 
Let $s$, $i \neq 0$, and $k$ be such that $h_1(s)=0$ and $h_1(s+1) = i \in W_k$.  
Let $t$ be such that $\xi_t \equiv \PR_{g_1}(\gn{\varphi})$.
Since $\PR_{g_1}(\gn{\varphi}) \notin \Img(f_1)$, we obtain $g_1(s+t) = \PR_{g_1}(\gn{\varphi})$. 
So $\PR_{g_1}(\gn{\PR_{g_1}(\gn{\varphi})})$ holds.
\end{proof}

We shall complete our proof of Theorem \ref{thmNP4}. 
The first clause of the theorem follows from Claims \ref{NP4-3} and \ref{NP4-4}.
The second clause of the theorem is proved by using Claims \ref{Prop:h_1}.3 and \ref{NP4-2} as in the proof of Theorem \ref{thmND4}.
\end{proof}

\begin{cor}[The arithmetical completeness of $\NP$]\label{AC_NP}
\begin{align*}
    \NP & = \bigcap \{ \PL(\PR_T) \mid \PR_T(x) \text{ is a provability predicate satisfying }  T \vdash \ConL_T \},\\
    & = \bigcap \{ \PL(\PR_T) \mid \PR_T(x) \text{ is a } \Sigma_1 \text{ provability predicate satisfying }  T \vdash \ConL_T \}.
\end{align*}
Moreover, there exists a $\Sigma_1$ provability predicate $\PR_T(x)$ of $T$ such that $\NP = \PL(\PR_T)$.
\end{cor}

\begin{cor}[The arithmetical completeness of $\NPF$]\label{AC_NP4}
\begin{align*}
    \NPF & = \bigcap \{ \PL(\PR_T) \mid \PR_T(x) \text{ satisfies } \D{3} \text{ and } T \vdash \ConL_T \},\\
    & = \bigcap \{ \PL(\PR_T) \mid \PR_T(x) \text{ is } \Sigma_1 \text{ and satisfies } \D{3} \text{ and } T \vdash \ConL_T \}.
\end{align*}
Moreover, there exists a $\Sigma_1$ provability predicate $\PR_T(x)$ of $T$ such that $\NPF = \PL(\PR_T)$.
\end{cor}

\section{Concluding remarks}

In this paper, we studied the four logics $\NP$, $\ND$, $\NPF$, and $\NDF$ based on $\N$ obtained by adding at least one of the principles $\mathsf{P}$, $\mathsf{D}$, and $\mathsf{4}$.
For these logics, we first established modal completeness and the finite frame property with respect to the relational semantics of Fitting, Marek, and Truszczy\'nski (Theorem \ref{thm:MC}).
We then proved their uniform arithmetical completeness by constructing suitable provability predicates (Theorems \ref{thmND4} and \ref{thmNP4}).

For $\NDF$, our proofs of both modal completeness and arithmetical completeness required methods different from the previous constructions. 
In the proof of modal completeness, we introduced a new method for reconstructing finite frames which preserves the property of being an $\ND$-frame while ensuring transitivity. 
In the proof of arithmetical completeness, we used Arai's construction, which provides Rosser provability predicates satisfying the requirements corresponding to $\NDF$. 
However, applying Arai's construction directly would yield a provability logic stronger than $\NDF$. 
The construction in Section~\ref{AC1} applies Arai's construction only to formulas lying outside the arithmetical interpretation to avoid this problem.

As summarized in Section~\ref{Summary}, the remaining open problems seem to require methods beyond those developed here.  
For the logics $\mathsf{EN4}$, $\mathsf{ENP4}$, and $\mathsf{ECN4}$, the main difficulty is that they are naturally treated by neighborhood semantics, where $\mathsf{4}$ is not simply a transitivity condition on a binary relation.  
For the logics $\mathsf{CN}$, $\mathsf{CNP}$, $\mathsf{CND}$, $\mathsf{CN4}$, and $\mathsf{CNP4}$, a better understanding of the corresponding relational semantics is still needed. 
These problems remain natural directions for future work.

\section*{Acknowledgments}

The authors would like to thank the anonymous referees for their helpful comments and suggestions. 
The first author was supported by JST SPRING, Grant Number JPMJSP2148. 
The second author was supported by JSPS KAKENHI Grant Number JP23K03200.

\bibliographystyle{plain}
\bibliography{refs}

\end{document}